\documentclass[letterpaper,final,10pt]{amsart}

\usepackage{inputenc}
\usepackage{cleveref}
\usepackage{amsmath}
\usepackage{amsfonts}
\usepackage{amssymb}
\usepackage{graphicx}
\usepackage{mathrsfs}
\usepackage[obeyspaces]{url}
\usepackage[all,cmtip,curve,knot]{xy}
\usepackage{tikz}
\usetikzlibrary{arrows}
\tikzstyle{block}=[draw opacity=0.7,line width=1.4cm]

\usepackage[all,cmtip]{xy}
\usepackage{algorithm}
\usepackage{algorithmic}
\usepackage{tikz}
\usetikzlibrary{shapes,arrows}
\newtheorem{theorem}{Theorem}

\newtheorem{definition}{Definition}

\newtheorem{lemma}{Lemma}

\numberwithin{equation}{section}
\numberwithin{theorem}{section}
\numberwithin{lemma}{section}
\numberwithin{corollary}{section}
\numberwithin{definition}{section}
\numberwithin{example}{section}
\numberwithin{remark}{section}
\numberwithin{property}{section}
\numberwithin{proposition}{section}

\newcommand{\lift}[3][\pi]{\xymatrix{#2 \ar@{~>}[r]_{#1} & #3}}

\begin{document}
\title[Topologically Controlled Model Order Reduction]{On Topologically Controlled Model Reduction for Discrete-Time Systems}
\author{Fredy Vides}
\address{Scientific Computing Innovation Center, Universidad Nacional Aut\'onoma de Honduras, Tegucigalpa, Honduras }

\email{fredy.vides@unah.edu.hn}



\date{\today}

\begin{abstract}
In this document the author proves that several problems in data-driven numerical approximation of dynamical systems in $\mathbb{C}^n$, can be reduced to the computation of a family of constrained matrix representations of elements of the
group algebra $\mathbb{C}[\mathbb{Z}/m]$ in $\mathbb{C}^{n\times n}$, factoring through the commutative
algebra $Circ(m)$ of circulant matrices in $\mathbb{C}^{m\times m}$, for some integers $m\leq n$.

The solvability of the previously described matrix representation problems is studied. Some
connections of the aforementioned results, with numerical analysis of dynamical systems, are outlined, a prorotypical algorithm for the computation of the matrix representations, and some numerical implementations of the algorithm, will be presented.
\end{abstract}

\maketitle

\section{Introduction}

In this document we will study discrete-time dynamical systems determined by the pair $(\Sigma,\{\Theta_t\})$, with $\Sigma\subseteq \mathbb{C}^n$, and where $\{\Theta_t|t\in \mathbb{Z}\}$ is a family of continuous functions from $\Sigma$ to $\Sigma$, such that $\Theta_t\circ\Theta_s=\Theta_{t+s}$ and $\Theta_0=\mathrm{id}_\Sigma$, for every pair of integers $t,s\geq 0$.

We will prove that several important problems in numerical analysis and data-driven discovery of discrete-time dynamical systems of the form $(\Sigma,\{\Theta_t\})$ in $\mathbb{C}^n$, can be reduced to the computation of a family of discrete-time transition matrices $\{F_t\}_{t=1}^{m-1} \subseteq \rho_m(\mathbb{C}[\mathbb{Z}/m]) \subset \mathbb{C}^{n\times n}$ of rank at most $m$ with $m\leq n$, for some matrix representation of the group algebra $\mathbb{C}[\mathbb{Z}/m]$, together with two matrices $\hat{K},\hat{T}\in \mathbb{C}^{n\times n}$ of rank at most $m$, that are related to some evolution history data $\{v_1,\ldots,v_m\}\subseteq \Sigma\subseteq \mathbb{C}^n$, (approximately) generated by the dynamical $(\Sigma,\{\Theta_t\})$, by the equations $v_{t+1}=\hat{K}F_t\hat{T}v_t$ for $1\leq t\leq m-1$.

We will also show that each variation of the problem corresponding to the computation of the
aforementioned transition matrices, can be reduced to solving a constrained matrix representation problem of the
group algebra $\mathbb{C}[\mathbb{Z}/m]$ in $\mathbb{C}^{n\times n}$, factoring through the commutative
algebra $Circ(m)$ of circulant matrices in $\mathbb{C}^{m\times m}$, for some integers $m\leq n$.

The motivation for the theoretical and computational machinery presented in this document came from some questions raised by M. H. Freedman along the lines of \cite{Freedman}, concerning to the implications in linear algebra and matrix computations of the so called Kirby Torus Trick, presented by R. Kirby in \cite{Kirby}.

We study the solvability of the previously described matrix representation problems. Some
connections of the aforementioned results, with numerical analysis of dynamical systems, are outlined, a prorotypical algorithm for the 
matrix representation computations, and some numerical implementations of the algorithm will be presented.

\section{Preliminaries and Notation}

Given two positive integers $p,q$ such that $p\geq q$, we will write $p~\mathrm{mod}~q$ to denote the integer $r$, such that $p=mq+r$ for some integer $m$. We will write $\mathbb{Z}^+_0$ to denote the set $\{z\in\mathbb{Z}|z\geq 0\}$.

Given $k\in \mathbb{Z}_0^+$, we will write $\mathbb{Z}/k$ to denote the (additive) cyclic group $\mathbb{Z}/k\mathbb{Z}=\{\hat{0},\hat{1},\hat{2},\ldots,\widehat{k-1}\}$.

Given any matrix $X\in \mathbb{C}^{m\times n}$, we will write $X_{ij}$ to denote the $ij$ entry of $X$, and we will write $X^\ast$ to denote its conjugate transpose 
$\overline{X^\top}=(\overline{X_{ji}})\in \mathbb{C}^{n\times m}$. We will identify elements in $\mathbb{C}^n$ with elements in $\mathbb{C}^{n\times 1}$. As a consequence of this identification, given $x,y\in \mathbb{C}^n$, $y^\ast x\in \mathbb{C}$ will determine the Ecuclidean inner product $\langle x,y\rangle$ in $\mathbb{C}^n$, while  $xy^\ast\in \mathbb{C}^{n\times n}$ will determine a rank-one matrix in $\mathbb{C}^{n\times n}$.

In this document we write $\mathbf{1}_n$ and $\mathbf{0}_{m\times n}$ to denote the identity and zero matrices in $\mathbb{C}^{n\times n}$ and $\mathbb{C}^{m\times n}$, respectively. We will write $\mathbf{0}_n$ to denote the zero matrix in $\mathbb{C}^{n\times n}$.

A set of $m$ elements $v_1,\ldots,v_m \in \mathbb{C}\backslash \{\mathbf{0}\}$ is said to be an orthogonal $m$-system if 
\begin{equation}
v_j^\ast v_k=\delta_{j,k} v_j^\ast v_k,
\label{orth-sys-cond-1}
\end{equation}
for $1\leq j,k\leq m$.

From here on, we will write $\delta_{k,j}$ to denote the Kronecker delta defined by
\begin{equation}
\delta_{k,j}=\left\{
\begin{array}{l}
1,\:\:\: k=j,\\
0,\:\:\: k\neq j.
\end{array}
\right.
\label{delta-def}
\end{equation}

We say that the set of vectors $v_1,\ldots,v_m \in \mathbb{C}\backslash \{\mathbf{0}\}$ is an orthonormal $m$-system if 
the vectors $v_j$ satisfy \eqref{orth-sys-cond-1} and in addition
\begin{equation}
\|v_j\|_2= \sqrt{v_j^\ast v_j}=1,
\label{orth-sys-cond-2}
\end{equation}
for $1\leq j\leq m$.

We will write $\hat{e}_{j,n}$ to denote the element in $\mathbb{C}^{n\times 1}$ represented by the expression.
\begin{equation}
\hat{e}_{j,n}=\begin{bmatrix}
\delta_{1,j}\\
\delta_{2,j}\\
\vdots\\
\delta_{k,j}\\
\vdots\\
\delta_{n,j}
\end{bmatrix}
\label{def-ejn}
\end{equation}
Each $\hat{e}_{j,n}$ can be interpreted as the $j$-column of $\mathbf{1}_n=[\hat{e}_{1,n} ~ \hat{e}_{2,n} ~ \cdots ~ \hat{e}_{n,n}]$.

A matrix $B\in \mathbb{C}^{n\times n}$ is said to be normal if $BB^\ast=B^\ast B$, a matrix $A\in \mathbb{C}^{n\times n}$ is said to be Hermitian if $X^\ast=X$, and a hermitian matrix $P$ is said to be an orthogonal projection or just a projection, if $P^2=P=P^\ast$.

A matrix $X\in \mathbb{C}^{n\times n}$ is said to be unitary if $X^\ast X=XX^\ast =\mathbf{1}_n$. A matrix $A\in \mathbb{C}^{n\times m}$ is said to be positive if there is a matrix $B\in \mathbb{C}^{n\times n}$ such that $A=B^\ast B$, we also write $A\geq 0$ to indicate that $A$ is positive. We will denote by $\mathbb{U}(n)$ and $\mathbb{P}(n)$, the sets of unitaries and positive matrices in $\mathbb{C}^{n\times n}$, respectively.

Given a matrix $W\in \mathbb{C}^{m\times n}$, we write $\mathrm{Ad}[W]$ to denote the linear map from $\mathbb{C}^{n\times n}$ to $\mathbb{C}^{m\times m}$, defined by the operation $\mathrm{Ad}[W](X)=WXW^\ast$ for any $X\in \mathbb{C}^{n\times n}$.

We say that a matrix $A\in \mathbb{C}^{n\times n}$ is invertible if there is one matrix $B\in\mathbb{C}^{n\times n}$ such that $AB=BA=\mathbf{1}_n$. We will write $\mathbb{GL}(n)$ to denote the set of invertible matrices in $\mathbb{C}^{n\times n}$. Given a matrix $X\in \mathbb{C}^{n\times n}$, we will write $\sigma(X)$ to denote the spectrum of $X$, that is the set $\{z\in \mathbb{C}|X-z\mathbf{1}_n\notin \mathbb{GL}(n)\}$.

From here on we write $\|\cdot\|_2$ to denote the Euclidean norm in $\mathbb{C}^n$ defined by the operation $\|x\|_2=\sqrt{x^\ast x}$ for any $x\in \mathbb{C}^n$. In this document we will write $\|\cdot\|$ to denote the spectral norm in $\mathbb{C}^{n\times n}$ defined by the opration $\|X\|=\sup_{\|x\|_2=1}\|Ax\|_2$, for any $A\in \mathbb{C}^{n\times n}$.

\begin{definition}
A linear map $\varphi:\mathbb{C}^{n\times n} \to \mathbb{C}^{m\times m}$ is said to be a completely positive (CP) linear map if $\varphi(A)\geq 0$ for every positive $A\in \mathbb{C}^{n\times n}$, and if it has a Choi's representation of the form 
$\varphi=\sum_{j=1}^k\mathrm{Ad}[W_j]$ for some matrices $W_j\in \mathbb{C}^{m\times n}$.
\end{definition}

We will write $CP(n,m)$ to denote the set of completely positve linear maps from $\mathbb{C}^{n\times n}$ to $\mathbb{C}^{m\times m}$.

Given any $X\in \mathbb{C}^{n\times n}$ and any $p\in\mathbb{C}[z]$ determined by the formula $p(z)=a_0+a_1z+a_2z^2+\cdots+a_nz^n$, we will write $p(X)$ to denote the matrix defined by the expression $p(X)=a_0\mathbf{1}_n+a_1X+a_2X^2+\cdots+a_nX^n$.
 
Given $X\in \mathbb{C}^{n\times n}$, we will write $\mathbb{C}[X]$ to denote the commutative algebra determined by the set $\{p(X)|p\in\mathbb{C}[z]\}=\mathrm{span}_{\mathbb{C}}\{\mathbf{1}_n,X,X^2,\ldots,X^{n-1}\}$, with respect to the usual addition and multiplication operations in $\mathbb{C}^{n\times n}$. 
We have that in fact $\mathbb{C}[X]$ is an algebra since, $X^m\in \mathrm{span}_{\mathbb{C}}\{\mathbf{1}_n,X,X^2,\ldots,X^{n-1}\}$ for every integer $m\geq n$, as a consequence of the Cayley-Hamilton Theorem, and that $\mathbb{C}[X]$ is commutative as a consequence of the identity $(aX^k)(bX^j)=abX^{k+j}=baX^{j+k}=bX^jaX^k$, that holds for each $a,b\in \mathbb{C}$ and for each pair of integers $k,j\geq 1$.

We will write $\mathrm{Circ}(k)$ to denote the commutative algebra of $k\times k$ Circulant matrices that is defined by the expression:
\begin{equation}
\mathrm{Circ}(k)=\mathbb{C}[C_k]=\{p(C_k)|p\in \mathbb{C}[z]\}
\label{eq:circulant_def}
\end{equation}
where $C_k$ is the cyclic permutation matrix defined as follows.
\begin{equation}
C_k=
\begin{bmatrix}
\mathbf{0}_{1\times(k-1)} & 1\\
\mathbf{1}_{k-1} & \mathbf{0}_{(k-1)\times 1}
\end{bmatrix}
\label{eq:cyclic_shift_def}
\end{equation}

Given a matrix $X\in \mathbb{C}^{n\times n}$, we write $Z(X)$ to denote the commutant set of $X$ defined by the expression $Z(X)=\{Y\in \mathbb{C}^{n\times n}|XY=YX\}$.

Given a finite group $G$, we will write $\mathbb{C}[G]$ to denote the group algebra over $\mathbb{C}$.

For a finite group $G$ In this document we will focus on group algebra representations of $\mathbb{C}[G]$ determined by algebra homomorphisms of the form $\rho:\mathbb{C}[G]\to \mathbb{C}^{n\times n},\sum_{g\in G}c_g g\mapsto \sum_{g\in G}c_g\rho(g)$, such that $\rho|_G$ is a group representation of $G$ in $\mathbb{GL}(n)$. We will say that an algebra representation $\rho:\mathbb{C}[G]\to \mathbb{C}^{n\times n}$ is unitary if $\rho|_G(G)\subset \mathbb{U}(n)$.

Given a discrete-time dynamical system $(\Sigma,\{\Theta_t\})$, if there is an integer $T>0$ such that 
$\Theta_{t+T}=\Theta_t$ for each every $t\in\mathbb{Z}_0^+$, we say that $(\Sigma,\{\Theta_t\})$ is a discrete-time $T$-periodic dynamical system.

Given a discrete-time dynamical system $(\Sigma,\{\Theta_t\})$, a set of vectors $\mathbb{H}[\Sigma,m]=\{v_1,\ldots,v_m\}\subseteq \Sigma$  will be called a $m$-system of history vectors for $(\Sigma,\{\Theta_t\})$, if they satisfy the relations $v_{k+1}$ $=$ $\Theta_1(v_k)$ $=$ $\Theta_k(v_1)$ for $0\leq k\leq m-1$.

Given $\delta,\varepsilon>0$, we will say that a discrete-time dynamical system $(\Sigma,\{\Theta_t\})$ is $(T,\delta,\varepsilon)$-almost-periodic dynamical system, if there is a $T$-periodic discrete-time dynamical system $(\tilde{\Sigma},\{\tilde{\Theta}_t\})$ such that for each $x\in\Sigma$ there is $\tilde{x}\in \tilde{\Sigma}$ such that $\|x-\tilde{x}\|\leq \delta$, and $\|\Theta_t(x)-\tilde{\Theta}_t(\tilde{x})\|_2\leq \varepsilon$ for each $t\in \mathbb{Z}_0^+$ and every $(x,\tilde{x})\in \Sigma\times \tilde{\Sigma}$ such that $\|x-\tilde{x}\|\leq \delta$.

\section{Topological Control Method (TCM)}
\label{sec:main}

\subsection{Switched Closed Loop Reduced Order Models SCL-ROM}

In this section we will stablish the notion of $\mathbb{S}^1$ topological control considered for this study.

Given a discrete-time dynamical system $(\Sigma,\{\Theta_t\})$ with $\Sigma\subseteq \mathbb{C}^{n\times n}$, and a $m$-system of history vectors $\mathbb{H}[\Sigma,m]=\{v_1,\ldots,v_m\}$ for $(\Sigma,\{\Theta_t\})$, we say that $(\Sigma$ $,$ $\{\Theta_t\}$ $,$ $\mathbb{H}[\Sigma,m])$ is {\bf \em topologically controlled} by a topological manifold $\mathbb{M}\subseteq \mathbb{C}$ or {\bf \em $\mathbb{M}$-controlled}, if there is a matrix $Z\in \mathbb{C}^{n\times n}$ with $\sigma(Z)\subseteq \mathbb{M}$, an algebra homomorphism $\varphi:\mathbb{C}[Z]\to \mathbb{C}^{n\times n}$, a family of polynomials $\{f_0,\ldots,f_{m-1}\}$, and two projections $K,T\in \mathbb{C}^{n\times n}$ such that $\Theta_k(v_1)=K\varphi(f_k(Z))Tv_1$, for each $k\geq 0$. We will call the $6$-tuple $(\mathbb{M},Z,K,T,\varphi,\{f_t\})$ a topological control for $(\Sigma,\{\Theta_t\},\mathbb{H}[\Sigma,m])$.

Given $\varepsilon>0$ and manifold $\mathbb{M}\subseteq \mathbb{C}$, and a $\mathbb{M}$-controlled discrete-time dynamical system $(\Sigma,\{\Theta_t\})$ with $\Sigma\subseteq \mathbb{C}^{n\times n}$, and a topological control $(\mathbb{M},Z,K,T$ $,$ $\varphi$ $,$ $\{f_t\})$ for $(\Sigma,\{\Theta_t\},\mathbb{H}[\Sigma,m])$, we say that $(\mathbb{M},Z,K,T,\varphi,\{f_t\})$ is a {\bf \em control of order $k$}, if there is an integer $k>0$, together with maps $\Pi_k\in Cp(n,k)$, $\Phi\in CP(k,n)$, such that $\|\Theta_k(v_1)-K\varphi(f_k(Z))T\tilde{v}\|\leq \varepsilon$ and $\|\varphi(X)-\Phi\circ \Pi_k(X)\|\leq \varepsilon$, for each $f_k$, each $X\in C[Z]$, and some $\tilde{v}\in \mathbb{C}^n$. In this case we say that $(\Sigma,\{\Theta_t\},\mathbb{H}[\Sigma,m])$ is {\bf \em $(\mathbb{M},k,\varepsilon)$-controlled}.

Given $\varepsilon>0$, a discrete-time dynamical system $(\Sigma,\{\Theta_t\})$ with $\Sigma\subseteq \mathbb{C}^{n\times n}$, and a $m$-system of history vectors $\mathbb{H}[\Sigma,m]=\{v_1,\ldots,v_m\}$ for $(\Sigma,\{\Theta_t\})$, we say that $(\Sigma,\{\Theta_t\},\mathbb{H}[\Sigma,m])$ is {\bf \em $\varepsilon$-approximately topologically controlled by $\mathbb{Z}/m$} or {\bf \em $(\mathbb{Z}/m,\varepsilon)$-controlled}, if there is a unitary representation $\rho_m:\mathbb{C}[\mathbb{Z}/m]\to \mathbb{C}^{m\times m}$, an algebra homomorphism $\varphi:\mathbb{C}^{m\times m}\to \mathbb{C}^{n\times n}$, a family of functions $\{f_0,\ldots,f_{m-1}\}$, and two projections $K,T\in \mathbb{C}^{n\times n}$ such that $\|\Theta_k(v_1)-K\varphi\circ\rho_m(f_k(\hat{1}))T\tilde{v}\|\leq \varepsilon$, for each $k\geq 0$ and some $\tilde{v}\in \mathbb{C}^n$, with $\hat{1}\in \mathbb{Z}/m$.

For a given a discrete-time dynamical system $(\Sigma,\{\Theta_t\})$ with $\Sigma\subseteq \mathbb{C}^{n\times n}$, and a $m$-system of history vectors $\mathbb{H}[\Sigma,m]=\{v_1,\ldots,v_m\}$ for $(\Sigma,\{\Theta_t\})$, we approach the local controllability of $(\Sigma,\{\Theta_t\})$ by computing a switched closed loop control system $(\hat{\Sigma},\{\hat{\Theta}_t\})$ in the sense of \cite[\S4.2,Example 4.2]{bilinear_systems}, determined by the decomposition.
\begin{equation}
\hat{\Sigma}:\left\{
\begin{array}{l}
\hat{v}_{1}=\alpha \hat{T}v_1\\
\hat{v}_{t+1}=\hat{F}_t\hat{v}_t\\
\tilde{v}_{t+1}=\beta\hat{K}\hat{v}_{t+1}
\end{array}
\right.,t\geq 0
\label{eq:loc_bil_rep_sigma}
\end{equation} 
for some $(\alpha,\beta,\hat{K},\{\hat{F}_t\},\hat{T})\in \mathbb{C}\times \mathbb{C}\times \mathbb{C}^{n\times n}\times \mathbb{C}^{n\times n}\times \mathbb{C}^{n\times n}$ to be determined. Given $\varepsilon>0$, the discrete-time system \eqref{eq:loc_bil_rep_sigma} is called a $\varepsilon$-approximate swithed closed-loop reduced order model ({\bf SCL-ROM}) of $(\Sigma,\{\Theta_t\})$, if $\|\hat{v}_t-\Theta_t(v_1)\|_2\leq \varepsilon$ for each $t\in\mathbb{Z}_0^+$.

\subsection{Some Connections with Dynamic Mode Decomposition}

Given $\varepsilon>0$, an integer $T>0$, and a discrete-time system $(\Sigma,\{\Theta_t\})$ in $\mathbb{C}^n$. Let us consider the evolution history determined by the difference equations.
\begin{equation}
\Sigma:\left\{
\begin{array}{l}
x_{t+1}=\Theta_t(x_t),\\
x_{0}=x
\end{array}
\right., t\in \mathbb{Z}_0^+
\label{thm_sigma}
\end{equation}
for some $x\in\Sigma$. 

Given $\varepsilon>0$. The computation of a {SCL-ROM} local $\varepsilon$-approximant $\tilde{\Sigma}$ of $\Sigma$, with respect to some sampled-data history $\{x_t\}$ of $\Sigma$, is related to the computations of closed-loop matrix realizations $H_t\in \mathbb{C}^{n\times n}$ and triples $(P_H(t),Q_H(t),F_H(t))\in \mathbb{C}^{n\times n}\times \mathbb{C}^{n\times n}\times \mathbb{C}^{n\times n}$ such that.
\begin{equation}
\left\{
\begin{array}{l}
\|(P_H(t)F_H(t)-H_tP_H(t))Q_H(t)\|\leq\varepsilon,\\
\|P_H(t) H_t-H_tP_H(t)\|=0,\\
\|P_H(t) x_{t}-x_{t}\|=0, \\
P_H(t)^2=P_H(t)=P_H(t)^\ast ,\\
Q_H(t)^2=Q_H(t)=Q_H(t)^\ast, \\
F_H(t)^\ast F_H(t)=F_H(t) F_H(t)^\ast
\end{array},0\leq t\leq T-1
\right.
\label{approx_constraints}
\end{equation}

The matrices $\{H_t\}$ in \eqref{approx_constraints} are determined by the connecting operator $\mathbb{K}$ for the sampled-data history $\{x_t\}$, in the sense of \cite[\S2]{DMD_Schmid} and \cite[\S2.2]{DMD_Kutz}, that satisfies the equations $x_{t+1}=\mathbb{K}x_t$, $0\leq t\leq T-1$, In particular we will consider $H_t=\mathbb{K}$, $t\in\mathbb{Z}_0^+$. The objective of topological control methods is to compute  matrix realizations $(F_H(t),K,\hat{T}(t))\in \mathbb{C}^{n\times n}\times \mathbb{C}^{n\times n}\times \mathbb{C}^{n\times n}$ such that: 
\begin{displaymath}
\|KF_H(t)\hat{T}(t)x_t-y_t\|\leq \varepsilon
\end{displaymath}
for $0\leq t\leq T-1$.

Given $\varepsilon>0$, an integer $T>0$, a family of vectors $\{x_t\}_{t=0}^{T-1}$ in $\mathbb{C}^{n}$ and matrices $H_t\in \mathbb{C}^{n\times n}$ determined by the connecting operator $\mathbb{K}$ for $\{x_t\}$.  For any triples $(P_H(t),Q_H(t),F_H(t))\in \mathbb{C}^{n\times n}\times \mathbb{C}^{n\times n}\times \mathbb{C}^{n\times n}$ that satisfy 
\eqref{approx_constraints}.

Let us consider the structured matrix equations in $\mathbb{C}^{2n\times 2n}$ determined by:
\begin{equation}
\left\{
\begin{array}{l}
(Q(t)X(t)-X(t)Q(t))P(t)=\mathbf{0}_{2n}\\
Q(t)^4=Q(t)^2\\
Q(t)^2=ZQ(t)=(Q(t)^2)^\ast\\
\end{array}
\right.,
0\leq t\leq T-1
\label{main_matrix_equation}
\end{equation}
where $P(t)$ and $Z$ have the form.
\begin{displaymath}
P(t)=\begin{bmatrix}
\mathbf{1}_n-P_H(t) & \mathbf{0}_n\\
\mathbf{0}_n & Q_H(t)
\end{bmatrix}
\:\:\:, \:\:\:
Z=\begin{bmatrix}
\mathbf{0}_n & \mathbf{1}_n\\
\mathbf{1}_n & \mathbf{0}_n
\end{bmatrix}
\end{displaymath}
If we set.
\begin{displaymath}
\hat{Q}(t)=\begin{bmatrix}
\mathbf{0}_n & P_H(t)\\
P_H(t) & \mathbf{0}_n
\end{bmatrix}\:\:\:,\:\:\:
\hat{X}(t)=\begin{bmatrix}
H_t & \mathbf{0}_n\\
\mathbf{0}_n & {F}_H(t)
\end{bmatrix}
\end{displaymath}
It can be seen that $(\hat{Q}(t),\hat{X}(t))\in (\mathbb{C}^{2n\times 2n})\backslash \{\mathbf{0}_{2n}\})^2$  are $\varepsilon$-approximate solvent pairs for \eqref{main_matrix_equation} in the sense that, $\|(\hat{Q}(t)\hat{X}(t)-\hat{X}(t)\hat{Q}(t))P(t)\|\leq\varepsilon$, $Q(t)^4=Q(t)^2$ and $Q(t)^2=ZQ(t)=(Q(t)^2)^\ast$ for $0\leq t\leq T-1$.

Given a positive integer $T$, together with $\varepsilon,\delta>0$. In this first paper on the subject of topological control of discrete-time systems, we focus on the computation of transition mappings $F_H(t)$, avoiding an explicit computation of the connecting operator $\mathbb{K}$, instead we compute the family $\{F_H(t)\}$ by solving some constrained representation problems for $\mathbb{C}[\mathbb{Z}/T]$ for a given integer $T>0$, restricting our attention to almost $(T,\varepsilon,\delta)$-periodic discrete-time systems.

\subsection{Main Objectives}

We prove the solvability of the problem of finding a $\varepsilon$-approximate SCL-ROM $\tilde{\Sigma}$ described by \eqref{eq:loc_bil_rep_sigma} for a discrete-time system $(\Sigma,\{\Theta_t\})$ with a $m$-system of history vectors $\mathbb{H}[\Sigma,m]=\{v_1,\ldots,v_m\}\subseteq \Sigma$, by computing matrix representations of $\mathbb{C}[\mathbb{Z}/T]$ such that the switching law of the family $\{F_t\}$ is controlled by some family $\{f_t\}\subset \mathbb{C}[\mathbb{Z}/T]$ subject to almost time-perdiodic constraints on $(\Sigma,\{\Theta_t\})$. We will then design and implement a prototypical numerical algorithm that numerically solves the aforementioned problems.

\subsection{Topologically Controlled Model Order Reduction}

\label{some_results}

Let us consider an orthogonal $m$-system $v_{1},\cdots,v_{m}\in\mathbb{C}^{n\times1}\backslash\{\mathbf{0}\}$ with
$m\leq n$. From here on, we will write $C[v_{1}|v_{m}]$ to denote the matrix
in $\mathbb{C}^{n\times n}$ defined by the equation.

\begin{equation}
C[v_{1}|v_{m}]=\frac{1}{v_{m}^{\ast}v_{m}}v_{1}v_{m}^{\ast}+\sum_{j=1}^{m-1}\frac{1}{v_{j}^{\ast}v_{j}}v_{j+1}v_{j}^{\ast}+\mathbf{1}_n-P[v_1|v_m]
\label{p-cycle-def}
\end{equation}
where $P[v_1|v_m]$ is the matrix in $\mathbb{C}^{n\times n}$
defined by the equation.

\begin{equation}
P[v_1|v_m]=\sum_{j=1}^{m}\frac{1}{v_{j}^{\ast}v_{j}}v_{j}v_{j}^{\ast}
\label{p-cycle-proj-def}
\end{equation}

\begin{lemma}
\label{First_Projective_lemma}
Given an orthogonal $m$-system $v_{1},\cdots,v_{m}\in\mathbb{C}^{n\times1}\backslash\{0\}$ with $m\leq n$. The matrix $P[v_1|v_m]$ defined by \eqref{p-cycle-proj-def} is an orthogonal projection such that $P[v_1|v_m]v_j$ $=$ $v_j$, $1\leq j\leq m$. Moreover, $\mathbf{1}_n-P[v_1|v_m]$ is an orthogonal projection such that $P[v_1|v_m](\mathbf{1}_n-P[v_1|v_m])=(\mathbf{1}_n-P[v_1|v_m])P[v_1|v_m]=\mathbf{0}_n$
\end{lemma}
\begin{proof}
Given an orthogonal $m$-system $v_{1},\cdots,v_{m}\in\mathbb{C}^{n\times1}\backslash\{0\}$ with $m\leq n$. For each $1\leq j\leq m$, let us set
\begin{equation}
V_j=\frac{1}{v_{j}^{\ast}v_{j}}v_{j}v_{j}^{\ast},
\label{Vj-def}
\end{equation}
it is clear that each $V_j$ satisfies the relation,
\begin{equation}
V_j^\ast=V_j,
\label{Vj-proj-ident-1}
\end{equation}
and we will also have that,
\begin{equation}
V_j^2=\frac{1}{v_{j}^{\ast}v_{j}}v_{j}v_{j}^{\ast}\frac{1}{v_{j}^{\ast}v_{j}}v_{j}v_{j}^{\ast}=\frac{v_{j}^{\ast}v_{j}}{(v_{j}^{\ast}v_{j})^2}v_jv_j^\ast=\frac{1}{v_{j}^{\ast}v_{j}}v_jv_j^\ast=V_j.
\label{Vj-proj-ident-2}
\end{equation}
By \eqref{Vj-proj-ident-1} and \eqref{Vj-proj-ident-2}, we will have that $V_j^\ast=V_j=V_j^2$, and this implies that each $V_j$ is projection, and by orthogonality ot the system $v_1,\ldots,v_m$, we have that
\begin{equation}
V_jV_k=\frac{1}{v_{j}^{\ast}v_{j}}v_{j}v_{j}^{\ast}\frac{1}{v_{k}^{\ast}v_{k}}v_{k}v_{k}^{\ast}=\frac{v_{j}^{\ast}v_{k}}{v_{j}^{\ast}v_{j}v_{k}^{\ast}v_{k}}v_jv_k^\ast=\frac{\delta_{j,k}v_{j}^{\ast}v_{k}}{v_{j}^{\ast}v_{j}v_{k}^{\ast}v_{k}}v_jv_k^\ast=\delta_{j,k}V_{j}V_k.
\label{Vj-proj-ident-3}
\end{equation}
This implies that the projections $V_1,\ldots,V_m$ are mutually orthogonal projections, and it can be seen that $(\sum_{j}V_j)^\ast=\sum_{j}V_j^\ast=\sum_{j}V_j$. Moreover, 
\begin{equation}
P[v_1|v_m]^2=(\sum_{j=1}^mV_j)^2=\sum_{j,k=1}^mV_jV_k=\sum_{j=1}^mV_j^2=\sum_{j=1}^mV_j=P[v_1|v_m].
\label{Vj-proj-ident-4}
\end{equation}
We also have that.
\begin{equation}
(P[v_1|v_m])^\ast=\left(\sum_{j=1}^mV_j\right)^\ast=\sum_{j=1}^mV_j^\ast=\sum_{j=1}^mV_j=P[v_1|v_m].
\label{Vj-proj-ident-5}
\end{equation}
By \eqref{p-cycle-proj-def} and by orthogonality of the system $v_1,\ldots,v_m$, we will have that for $1\leq j\leq m$,
\begin{eqnarray}
P[v_1|v_m]v_j=\sum_{k=1}^{m}\frac{1}{v_{k}^{\ast}v_{k}}v_{k}(v_{k}^{\ast}v_j)=\sum_{k=1}^{m}\frac{1}{v_{k}^{\ast}v_{k}}v_{k}(\delta_{k,j}v_{k}^{\ast}v_j)=\frac{1}{v_j^\ast v_j}v_j(v_j^\ast v_j)=v_j.
\label{Vj-proj-ident-6}
\end{eqnarray}
By \eqref{Vj-proj-ident-4} we will have that.
\begin{eqnarray}
(\mathbf{1}_n-P[v_1|v_m])^2&=&\mathbf{1}_n-2P[v_1|v_m]+P[v_1|v_m]^2\nonumber\\
&=&\mathbf{1}_n-2P[v_1|v_m]+P[v_1|v_m]=\mathbf{1}_n-P[v_1|v_m]
\label{Vj-proj-ident-7}
\end{eqnarray}
As a consequence of \eqref{Vj-proj-ident-5} we will also have that.
\begin{equation}
(\mathbf{1}_n-P[v_1|v_m])^\ast=\mathbf{1}_n-P[v_1|v_m]^\ast=\mathbf{1}_n-P[v_1|v_m]
\label{Vj-proj-ident-8}
\end{equation}
By \eqref{Vj-proj-ident-4} we will have that,
\begin{equation}
(\mathbf{1}_n-P[v_1|v_m])P[v_1|v_m]=P[v_1|v_m]-P[v_1|v_m]^2=P[v_1|v_m]-P[v_1|v_m]=\mathbf{0}_n
\label{Vj-proj-ident-9}
\end{equation}
and also that.
\begin{equation}
P[v_1|v_m](\mathbf{1}_n-P[v_1|v_m])=P[v_1|v_m]-P[v_1|v_m]^2=P[v_1|v_m]-P[v_1|v_m]=\mathbf{0}_n
\label{Vj-proj-ident-10}
\end{equation}
This completes the proof.
\end{proof}

\begin{lemma}
\label{lema-tool-2}
Given an orthogonal $m$-system $v_{1},\cdots,v_{m}\in\mathbb{C}^{n\times1}\backslash\{0\}$ with $m\leq n$. The matrices $C[v_{1}|v_{m}]$ and $P[v_1|v_m]$ defined by \eqref{p-cycle-def} and \eqref{p-cycle-proj-def}, respectively, satisfy the following conditions:
\begin{itemize}
\item $C[v_{1}|v_{m}]v_j=v_{j+1}$, $1\leq j\leq m-1$,
\item $C[v_{1}|v_{m}]v_m=v_1$, 
\item $C[v_{1}|v_{m}]P[v_1|v_m]=P[v_1|v_m]C[v_{1}|v_{m}]=P[v_1|v_m]C[v_{1}|v_{m}]P[v_1|v_m]$,
\item $P[v_1|v_m]C[v_{1}|v_{m}]P[v_1|v_m]=\frac{1}{v_m^\ast v_m}v_1v_m^\ast+\sum_{j=1}^{m-1}\frac{1}{v_j^\ast v_j}v_{j+1}v_j^\ast$,
\item $C[v_{1}|v_{m}](\mathbf{1}_n-P[v_1|v_m])=(\mathbf{1}_n-P[v_1|v_m])C[v_{1}|v_{m}]=\mathbf{1}_n-P[v_1|v_m]$.
\end{itemize}
\end{lemma}
\begin{proof}
Since $C[v_1|v_m]$ is defined by \eqref{p-cycle-def}, by \cref{First_Projective_lemma}, we will have that for $1\leq j\leq m$:
\begin{eqnarray}
C[v_1|v_m]v_j&=&\frac{1}{v_m^\ast v_m}v_1(v_m^\ast v_j)+\sum_{j=1}^{m-1}\frac{1}{v_k^\ast v_k}v_{k+1}(v_{k}^\ast v_j)+v_j-P[v_1|v_m]v_j\nonumber\\
&=&\frac{\delta_{m,j}v_m^\ast v_j}{v_m^\ast v_m}v_1+\sum_{j=1}^{m-1}\frac{\delta_{k,j}v_k^\ast v_j}{v_k^\ast v_k}v_{k+1}+\mathbf{0}\nonumber\\
&=&\left\{
\begin{array}{l}
v_{j+1},\:\:\: 1\leq j\leq m-1,\\
v_1, \:\:\: j=m.
\end{array}
\right.
\end{eqnarray}

By \cref{First_Projective_lemma}, and by \eqref{p-cycle-def} and \eqref{p-cycle-proj-def}, on one hand we will have that,
\begin{eqnarray}
P[v_1|v_m]C[v_1|v_m]&=&P[v_1|v_m](\frac{1}{v_m^\ast v_m}v_1v_m^\ast+\sum_{j=1}^{m-1}\frac{1}{v_j^\ast v_j}v_{j+1}v_{j}^\ast+\mathbf{1}_n-P[v_1|v_m])\nonumber\\
&=&\frac{1}{v_m^\ast v_m}(P[v_1|v_m]v_1)v_m^\ast+\sum_{j=1}^{m-1}\frac{1}{v_j^\ast v_j}(P[v_1|v_m]v_{j+1})v_{j}^\ast\nonumber\\
&&+P[v_1|v_m](\mathbf{1}_n-P[v_1|v_m])\nonumber\\
&=&\frac{1}{v_m^\ast v_m}v_1v_m^\ast+\sum_{j=1}^{m-1}\frac{1}{v_j^\ast v_j}v_{j+1}v_{j}^\ast+\mathbf{0}_n\nonumber\\
&=&\frac{1}{v_m^\ast v_m}v_1v_m^\ast+\sum_{j=1}^{m-1}\frac{1}{v_j^\ast v_j}v_{j+1}v_{j}^\ast
\label{P1m-cond-1}
\end{eqnarray}
on the other hand we will have that.
\begin{eqnarray}
C[v_1|v_m]P[v_1|v_m]&=&(\frac{1}{v_m^\ast v_m}v_1v_m^\ast+\sum_{j=1}^{m-1}\frac{1}{v_j^\ast v_j}v_{j+1}v_{j}^\ast+\mathbf{1}_n-P[v_1|v_m])P[v_1|v_m]^\ast\nonumber\\
&=&\frac{1}{v_m^\ast v_m}v_1(v_m^\ast P[v_1|v_m]^\ast)+\sum_{j=1}^{m-1}\frac{1}{v_j^\ast v_j}v_{j+1}(v_{j}^\ast P[v_1|v_m]^\ast)\nonumber\\
&&+(\mathbf{1}_n-P[v_1|v_m])P[v_1|v_m]^\ast\nonumber\\
&=&\frac{1}{v_m^\ast v_m}v_1(P[v_1|v_m]v_m)^\ast+\sum_{j=1}^{m-1}\frac{1}{v_j^\ast v_j}v_{j+1}(P[v_1|v_m]v_j)^\ast+\nonumber\\
&&+(\mathbf{1}_n-P[v_1|v_m])P[v_1|v_m]\nonumber\\
&=&\frac{1}{v_m^\ast v_m}v_1v_m^\ast+\sum_{j=1}^{m-1}\frac{1}{v_j^\ast v_j}v_{j+1}v_j^\ast+\mathbf{0}_n\nonumber\\
&=&\frac{1}{v_m^\ast v_m}v_1v_m^\ast+\sum_{j=1}^{m-1}\frac{1}{v_j^\ast v_j}v_{j+1}v_j^\ast
\label{P1m-cond-2}
\end{eqnarray}
By combining \eqref{P1m-cond-1} and \eqref{P1m-cond-2} we have that 
\begin{equation}
P[v_1|v_m]C[v_1|v_m]=C[v_1|v_m]P[v_1|v_m]
\label{P1m-cond-3}
\end{equation}
Since $P[v_1|v_m]$ is an orthogonal projection, we will also have that.
\begin{equation}
P[v_1|v_m]C[v_1|v_m]=P[v_1|v_m]P[v_1|v_m]C[v_1|v_m]=P[v_1|v_m]C[v_1|v_m]P[v_1|v_m]
\label{P1m-cond-4}
\end{equation}
By \eqref{P1m-cond-4} and \eqref{P1m-cond-1}, it can be seen that $C[v_1|v_m]$ can be represented in the form.
\begin{equation}
C[v_1|v_m]=P[v_1|v_m]C[v_1|v_m]P[v_1|v_m]+\mathbf{1}_n-P[v_1|v_m]
\label{P1m-cond-5}
\end{equation}
By \cref{First_Projective_lemma}, we will have that,
\begin{eqnarray}
(\mathbf{1}_n-P[v_1|v_m])C[v_1|v_m]&=&(\mathbf{1}_n-P[v_1|v_m])P[v_1|v_m]C[v_1|v_m]P[v_1|v_m]\nonumber\\
&&+(\mathbf{1}_n-P[v_1|v_m])^2\nonumber\\
&=&\mathbf{0}_nC[v_1|v_m]P[v_1|v_m]+\mathbf{1}_n-P[v_1|v_m]\nonumber\\
&=&\mathbf{1}_n-P[v_1|v_m]
\label{P1m-cond-6}
\end{eqnarray}
and also that.
\begin{eqnarray}
C[v_1|v_m](\mathbf{1}_n-P[v_1|v_m])&=&P[v_1|v_m]C[v_1|v_m]P[v_1|v_m](\mathbf{1}_n-P[v_1|v_m])\nonumber\\
&&+(\mathbf{1}_n-P[v_1|v_m])^2\nonumber\\
&=&P[v_1|v_m]C[v_1|v_m]\mathbf{0}_n+\mathbf{1}_n-P[v_1|v_m]\nonumber\\
&=&\mathbf{1}_n-P[v_1|v_m]
\label{P1m-cond-7}
\end{eqnarray}
This completes the proof.
\end{proof}

\begin{lemma}
\label{min-pol-C1m}
Given an orthogonal $m$-system $v_{1},\cdots,v_{m}\in\mathbb{C}^{n\times1}\backslash\{0\}$ with $m\leq n$. The matrix $C[v_{1}|v_{m}]$ defined by \eqref{p-cycle-def} satisfies the equation
\begin{equation}
C[v_{1}|v_{m}]^m=\mathbf{1}_n.
\label{def-eq-C1m}
\end{equation}
Moreover, $p(z)=z^m-1$ is the minimal polynomial of $C[v_1|v_m]$.
\end{lemma}
\begin{proof}
Given an orthogonal $m$-system $v_{1},\cdots,v_{m}\in\mathbb{C}^{n\times1}\backslash\{0\}$ with $m\leq n$. By \eqref{p-cycle-def}, \eqref{p-cycle-proj-def}, and by iterating on \cref{lema-tool-2}, we will have that.
\begin{eqnarray}
C[v_{1}|v_{m}]^m&=&C[v_{1}|v_{m}]^{m-1}\left(\frac{1}{v_m^\ast v_m}v_1v_m^\ast+\sum_{j=1}^{m-1}\frac{1}{v_j^\ast v_j}v_{j+1}v_{j}^\ast+\mathbf{1}_n-P[v_1|v_m]\right)\nonumber\\
&=&\frac{1}{v_m^\ast v_m}(C[v_{1}|v_{m}]^{m-1}v_1)v_m^\ast+\sum_{j=1}^{m-1}\frac{1}{v_k^\ast v_j}(C[v_{1}|v_{m}]^{m-1}v_{j+1})v_{j}^\ast\nonumber\\
&&+~C[v_{1}|v_{m}]^{m-1}(\mathbf{1}_n-P[v_1|v_m])\nonumber\\
&=&\frac{1}{v_m^\ast v_m}v_{m-1+1}v_m^\ast+\sum_{j=1}^{m-1}\frac{1}{v_j^\ast v_{j}}v_{(j+1+(m-1))~\mathrm{mod}~m}v_{j}^\ast+\mathbf{1}_n-P[v_1|v_m]\nonumber\\
&=&\frac{1}{v_m^\ast v_m}v_{m}v_m^\ast+\sum_{j=1}^{m-1}\frac{1}{v_j^\ast v_{j}}v_{j}v_{j}^\ast+\mathbf{1}_n-P[v_1|v_m]\nonumber\\
&=&\sum_{j=1}^{m}\frac{1}{v_j^\ast v_{j}}v_{j}v_{j}^\ast+\mathbf{1}_n-P[v_1|v_m]\nonumber\\
&=&P[v_1|v_m]+\mathbf{1}_n-P[v_1|v_m]=\mathbf{1}_n.
\label{eq:min_pol_C1m}
\end{eqnarray}
By orthogonality properties we have that the system $v_1,\ldots,v_m$ is linearly independent. This fact combined  with \eqref{eq:min_pol_C1m} implies that $p(z)=z^m-1$ is the minimal polynomial of $C[v_1|v_m]$, since if $C[v_1|v_m]^k=\mathbf{1}_n$ for some $k<m$, we would have that $v_1=v_k$, which contradicts the linear independence of the system.
\end{proof}

\begin{lemma}
\label{unitary-C1m}
Given an orthonormal $m$-system $\{v_1,\ldots,v_m\}\in \mathbb{C}^n\backslash \{\mathbf{0}\}$ with $m\leq n$, we will have that the corresponding matrix $C[v_1|v_m]$ is unitary.
\end{lemma}
\begin{proof}
Given an orthonormal $m$-system $\{v_1,\ldots,v_m\}\in \mathbb{C}^n\backslash \{\mathbf{0}\}$, since $v_j^\ast v_j=1$ for each $1\leq j\leq m$, by \eqref{p-cycle-def} we will have that the matrix $C[v_1|v_m]$ satisfies the equation,
\begin{equation}
C[v_{1}|v_{m}]=v_{1}v_{m}^{\ast}+\sum_{j=1}^{m-1}v_{j+1}v_{j}^{\ast}+\mathbf{1}_n-P[v_1|v_m]
\label{C-1m-orthn-system}
\end{equation}
and also that the matrix $P[v_1,v_m]$ satisfies the equation.
\begin{equation}
P[v_{1}|v_{m}]=\sum_{j=1}^{m}v_{j}v_{j}^{\ast}
\label{P-1m-orthn-system}
\end{equation}
By \cref{lema-tool-2} we will have that,
\begin{eqnarray}
C[v_{1}|v_{m}]^\ast(\mathbf{1}_n-P[v_1|v_m])&=&((\mathbf{1}_n-P[v_1|v_m])C[v_{1}|v_{m}])^\ast\nonumber\\
&=&\mathbf{1}_n-P[v_1|v_m]
\label{C-1m-complement-1}
\end{eqnarray}
and also that.
\begin{eqnarray}
C[v_{1}|v_{m}]^\ast P[v_1|v_m]&=&(P[v_1|v_m]C[v_{1}|v_{m}])^\ast\nonumber\\
&=&(P[v_1|v_m]C[v_{1}|v_{m}]P[v_1|v_m])^\ast\nonumber\\
&=&v_{m}v_{1}^{\ast}+\sum_{j=1}^{m-1}v_{j}v_{j+1}^{\ast}
\label{C-1m-complement-2}
\end{eqnarray}
Since 
\begin{eqnarray*}
C[v_1|v_m]^\ast&=&C[v_1|v_m]^\ast(P[v_1|v_m]+\mathbf{1}_n-P[v_1|v_m])\\
&=&C[v_1|v_m]^\ast P[v_1|v_m]+\mathbf{1}_n-P[v_1|v_m]
\end{eqnarray*}
and 
\begin{eqnarray*}
C[v_1|v_m]&=&(P[v_1|v_m]+\mathbf{1}_n-P[v_1|v_m])C[v_1|v_m]\\
&=&P[v_1|v_m]C[v_1|v_m]+\mathbf{1}_n-P[v_1|v_m],
\end{eqnarray*}
by \eqref{C-1m-orthn-system} and \eqref{P-1m-orthn-system} we will have that.
\begin{eqnarray}
C[v_1|v_m]^\ast C[v_1|v_m]&=&C[v_1|v_m]^\ast P[v_1|v_m]P[v_1|v_m]C[v_1|v_m]+(\mathbf{1}_n-P[v_1|v_m])^2\nonumber\\
&=&\left(v_{m}v_{1}^{\ast}+\sum_{j=1}^{m-1}v_{j}v_{j+1}^{\ast}\right)\left(v_1v_m^\ast+\sum_{j=1}^{m-1}v_{j+1}v_j^\ast\right)\nonumber\\
&&+\mathbf{1}_n-P[v_1|v_m]\nonumber\\
&=&P[v_1|v_m]+\mathbf{1}_n-P[v_1|v_m]=\mathbf{1}_n
\end{eqnarray}
This implies that.
\begin{eqnarray}
C[v_1|v_m] C[v_1|v_m]^\ast&=(C[v_1|v_m]^\ast C[v_1|v_m])^\ast=(\mathbf{1}_n)^\ast=\mathbf{1}_n
\end{eqnarray}
This completes the proof.
\end{proof}

\begin{lemma}
\label{orthonormalization-lemma}
Given $m$ vectors $v_1,\ldots,v_m\in \mathbb{C}^{n}\backslash\{\mathbf{0}\}$ such that $2m\leq n$, there is a orthonormal $m$-system $\hat{v}_1,\ldots,\hat{v}_m\in \mathbb{C}^n$, two scalars $\rho,\kappa=\mathbb{C}$, two projections $K$ , $T$ and a unitary $U[v_1|v_m]$ in $\mathbb{C}^{n\times n}$  such that:
\begin{equation}
\left\{
\begin{array}{l}
Tv_1=\rho\hat{v}_1,\\
U[v_1|v_m]\hat{v}_j=\hat{v}_{j+1},\\
U[v_1|v_m]\hat{v}_m=\hat{v}_{1}\\
K\kappa\hat{v}_j=v_j
\end{array}
\right.
\label{cyclic-constraints}
\end{equation}
for each $1\leq j\leq m-1$.
\end{lemma}
\begin{proof}
Let us consider the matrix $\mathbb{H}[v_1|v_m]\in \mathbb{C}^{n\times m}$ defined by the expression.
\begin{equation}
\mathbb{H}[v_1|v_m]=
\begin{bmatrix}
| & | & & |\\
v_1 & v_2 & \cdots & v_m\\
| & | &  & |
\end{bmatrix}
\label{history_matrix}
\end{equation}
By the singular value decomposition theorem, we have that $\mathbb{H}[v_1|v_m]$ has a representation of the form.
\begin{equation}
\mathbb{H}[v_1|v_m]=
\begin{bmatrix}
| & | & & |\\
V_1 & V_2 & \cdots & V_m\\
| & | &  & |
\end{bmatrix}
\begin{bmatrix}
s_1 & 0& \cdots& 0\\
0 & s_2 & &\vdots \\
\vdots &  & \ddots  & 0\\
0 & \cdots&0&s_m
\end{bmatrix}
\begin{bmatrix}
| & | & & |\\
W_1 & W_2 & \cdots & W_m\\
| & | &  & |
\end{bmatrix}
\label{history_matrix_svd-1}
\end{equation}
Where $V_1,\ldots,V_m$ and $W_1,\ldots,W_m$ are orthonormal $m$-systems in $\mathbb{C}^n$ and $\mathbb{C}^m$ respectively, and with $s_1\geq s_2\geq \cdots s_m\geq 0$. Since $2m\leq n$, by the Gram-Schmidt orthonormalization theorem, we will have that there is an orthonormal $m$-system $U_1,\ldots,U_m\in (\mathrm{span}_{\mathbb{C}}\{V_1,\ldots,V_m\})^\perp$. Since $\{v_1,\ldots,v_m\}\in \mathbb{C}^n\backslash \{\mathbf{0}\}$ we will have that for each $1\leq j\leq m$, $s_1\geq  s_j>0$. Let us set $t_j=\sqrt{1-(s_j/s_1)^2}$, for $1\leq j\leq m$. We will have that $0\leq t_j\leq 1$ and $(s_j/s_1)^2+t_j^2=1$ for each $1\leq j\leq m$, since $s_j/s_1\leq 1$ for every $1\leq j\leq m$. Let us define the matrix $\mathbb{CH}[v_1|v_m]$ by the expression.
\begin{equation}
\mathbb{CH}[v_1|v_m]=
\begin{bmatrix}
| & | & & |\\
V_1 & V_2 & \cdots & V_m\\
| & | &  & |
\end{bmatrix}
\begin{bmatrix}
1 & 0& \cdots& 0\\
0 & s_2/s_1 & &\vdots \\
\vdots &  & \ddots  & 0\\
0 & \cdots&0&s_m/s_1
\end{bmatrix}
\begin{bmatrix}
| & | & & |\\
W_1 & W_2 & \cdots & W_m\\
| & | &  & |
\end{bmatrix}
\label{history_matrix_svd-2}
\end{equation}
and 
\begin{equation}
\mathbb{SH}[v_1|v_m]=
\begin{bmatrix}
| & | & & |\\
U_1 & U_2 & \cdots & U_m\\
| & | &  & |
\end{bmatrix}
\begin{bmatrix}
t_1 & 0& \cdots& 0\\
0 & t_2 & &\vdots \\
\vdots &  & \ddots  & 0\\
0 & \cdots&0&t_m
\end{bmatrix}
\begin{bmatrix}
| & | & & |\\
W_1 & W_2 & \cdots & W_m\\
| & | &  & |
\end{bmatrix},
\label{history_matrix_svd-3}
\end{equation}
Let us define the matrix $\hat{\mathbb{V}}=[\hat{v}_1 ~ \cdots ~ \hat{v}_m]\in \mathbb{C}^{n\times n}$ by the expression.
\begin{equation}
\hat{\mathbb{V}}=\mathbb{CH}[v_1|v_m]+\mathbb{SH}[v_1|v_m]
\label{history_matrix_svd-4}
\end{equation}
Since $(s_j/s_1)^2+t_j^2=1$ for each $1\leq j\leq m$, by \eqref{history_matrix_svd-2} and \eqref{history_matrix_svd-3} and by orthogonality of the $2m$-system $V_1,\ldots,V_m,U_1,\ldots,U_m$, we will have that $\hat{\mathbb{V}}^\ast\hat{\mathbb{V}}=\mathbf{1}_m$. This implies that $\hat{v}_1,\ldots,\hat{v}_m$ is an orthonormal $m$-system. By \cref{unitary-C1m} we have that $C[\hat{v}_1|\hat{v}_m]$ is a unitary matrix that satisfies the constraints $C[\hat{v}_1|\hat{v}_m]\hat{v}_j=\hat{v}_{j+1}$, $1\leq j\leq m-1$, and $C[\hat{v}_1|\hat{v}_m]\hat{v}_m=\hat{v}_{1}$.

Let us set.
\begin{eqnarray}
\kappa&=&s_1\nonumber\\
\rho&=&\frac{1}{s_1}(v_1^\ast v_1)\nonumber\\
K&=&\begin{bmatrix}
| & | & & |\\
V_1 & V_2 & \cdots & V_m\\
| & | &  & |
\end{bmatrix}
\begin{bmatrix}
| & | & & |\nonumber\\
V_1 & V_2 & \cdots & V_m\\
| & | &  & |
\end{bmatrix}^\ast\nonumber\\
T&=&\hat{v}_1\hat{v}_1^\ast\nonumber\\
U[v_1|v_m]&=&C[\hat{v}_1|\hat{v}_m]
\label{triple_def}
\end{eqnarray}

Since $\hat{\mathbb{V}}^\ast \hat{\mathbb{V}}=\mathbf{1}_n$ we will have that $K^2=\hat{\mathbb{V}}^\ast \hat{\mathbb{V}}\hat{\mathbb{V}}^\ast \hat{\mathbb{V}}=\hat{\mathbb{V}}^\ast \hat{\mathbb{V}}$ and $K^\ast=(\hat{\mathbb{V}}^\ast \hat{\mathbb{V}})^\ast=\hat{\mathbb{V}}^\ast \hat{\mathbb{V}}=K$.

Since $U_1,\ldots,U_m\in (\mathrm{span}_{\mathbb{C}}\{V_1,\ldots,V_m\})^\perp$, we will have that $K\hat{\mathbb{V}}=\mathbb{CH}[v_1|v_m]$, by \eqref{history_matrix_svd-2} this implies that.
\[
K\kappa\hat{v}_j=\kappa K\hat{\mathbb{V}}\hat{e}_{j,m}=s_1\mathbb{CH}[v_1|v_m]\hat{e}_{j,m}=\mathbb{H}[v_1|v_m]\hat{e}_{j,m}=v_j
\]

We will first show that $\hat{v}_1^\ast v_1\neq 0$, in fact, since $\mathbb{H}[v_1|v_m]=s_1\mathbb{CH}[v_1|v_m]$, and $\mathbb{SH}[v_1|v_m]^\ast\mathbb{CH}[v_1|v_m]=\mathbf{0}_n$ by orthogonality of $V_1,\ldots,V_m,U_1,\ldots,U_m$, we will have that.

\begin{eqnarray}
\hat{v}_1^\ast v_1&=&(\hat{\mathbb{V}}\hat{e}_{1,m})^\ast \mathbb{H}[v_1|v_m]\hat{e}_{1,m}\nonumber\\
&=&\hat{e}_{1,m}^\ast \hat{\mathbb{V}}^\ast \mathbb{H}[v_1|v_m]\hat{e}_{1,m}\nonumber\\
&=&\hat{e}_{1,m}^\ast \frac{1}{s_1}\mathbb{H}[v_1|v_m]^\ast \mathbb{H}[v_1|v_m]\hat{e}_{1,m}\nonumber\\
&=&\frac{1}{s_1}\hat{e}_{1,m}^\ast \mathbb{H}[v_1|v_m]^\ast \mathbb{H}[v_1|v_m]\hat{e}_{1,m}\nonumber\\
&=&\frac{1}{s_1}(\mathbb{H}[v_1|v_m]\hat{e}_{1,m})^\ast \mathbb{H}[v_1|v_m]\hat{e}_{1,m}\nonumber\\
&=&\frac{1}{s_1} v_1^\ast v_1>0
\label{non-zero-product}
\end{eqnarray}

We will also have that $T^2=\hat{v}_1\hat{v}_1^\ast\hat{v}_1\hat{v}_1^\ast=\hat{v}_1\hat{v}_1^\ast=T$ and $T^\ast =(\hat{v}_1\hat{v}_1^\ast)^\ast=\hat{v}_1\hat{v}_1^\ast=T$.

Since $\hat{v}_1^\ast v_1\neq 0$, by \eqref{non-zero-product} we have that $Tv_1=\hat{v}_1(\hat{v}_1^\ast v_1)=\frac{1}{s_1} (v_1^\ast v_1)\hat{v}_1=\rho \hat{v}_1$. This completes the proof.
\end{proof}

\begin{lemma}
\label{projective-morphism}
Given an orthonormal $m$-system $\hat{v}_1,\ldots,\hat{v}_m\in \mathbb{C}^n\backslash \{\mathbf{0}\}$ with $m\leq n$. There is $\hat{V}\in \mathbb{C}^{n\times m}$ determined by $\hat{v}_1,\ldots,\hat{v}_m$, such that the map $\Pi_m=\mathrm{Ad}[\hat{V}^\ast]\in CP(n,m)$ from $Z(P[\hat{v}_1|\hat{v}_m])$ onto $\mathbb{C}^{m\times m}$ preserves products in $Z(P[\hat{v}_1|\hat{v}_m])$, with $P[\hat{v}_1|\hat{v}_m]$ determined by \eqref{p-cycle-proj-def}. Moreover, we will have that $\Pi_m(P[\hat{v}_1|\hat{v}_m])=\mathbf{1}_m$, and $\Pi_m(X)=\Pi_m(P[\hat{v}_1|\hat{v}_m]X)$ for any $X\in \mathbb{C}^{n\times n}$.
\end{lemma}
\begin{proof}
Let us set.
\begin{equation}
\hat{V}=\sum_{j=1}^m \hat{v}_j\hat{e}_{j,m}^\ast=\begin{bmatrix}
| & | & & |\\
\hat{v}_1 & \hat{v}_2 & \cdots & \hat{v}_m\\
| & | & & |
\end{bmatrix}
\label{Pim-matrix-def}
\end{equation}
Since $\hat{v}_1,\ldots,\hat{v}_m\in \mathbb{C}^n$ is an orthonormal $m$-system, we will have that $\hat{V}^\ast \hat{V}=\mathbf{1}_m$, by \eqref{p-cycle-proj-def} we will have that.
\begin{equation}
P[\hat{v}_1|\hat{v}_m]=\sum_{j=1}^{m}\hat{v}_j\hat{v}_j^\ast=\hat{V}\hat{V}^\ast
\label{alternative-central-projection-def}
\end{equation}
Since $\hat{V}^\ast \hat{V}=\mathbf{1}_m$,  by \eqref{alternative-central-projection-def} we will have that,
\begin{equation}
\Pi_m(P[\hat{v}_1|\hat{v}_m])=\hat{V}^\ast \hat{V} \hat{V}^\ast \hat{V}=\mathbf{1}_n^2=\mathbf{1}_n
\label{unit-preserving-identity}
\end{equation}
we will also have that for any $X\in \mathbb{C}^{n\times n}$. 
\begin{eqnarray}
\Pi_m(XY)&=&\hat{V}^\ast \hat{V}\Pi_m(X)\nonumber\\
&=&\hat{V}^\ast \hat{V} \hat{V}^\ast X\hat{V}\nonumber\\
&=&\hat{V}^\ast P[\hat{v}_1|\hat{v}_m] X\hat{V}\nonumber\\
&&\Pi_m(P[\hat{v}_1|\hat{v}_m]X)
\label{first-restriction-identity-Pim}
\end{eqnarray}
By \eqref{alternative-central-projection-def} and \eqref{first-restriction-identity-Pim} we will have that for any two $X,Y\in Z(P[\hat{v}_1|\hat{v}_m])$. 
\begin{eqnarray}
\Pi_m(XY)&=&\Pi_m(P[\hat{v}_1|\hat{v}_m]XY)\nonumber\\
&=&\Pi_m(X P[\hat{v}_1|\hat{v}_m] Y)\nonumber\\
&=&\hat{V}^\ast X P[\hat{v}_1|\hat{v}_m] Y\hat{V}\nonumber\\
&=&\hat{V}^\ast X \hat{V} \hat{V}^\ast Y\hat{V}\nonumber\\
&=&\Pi_m(X)\Pi_m(Y)
\label{prod-preserving-ident-Pim}
\end{eqnarray}
By \eqref{Pim-matrix-def} and by orthonormality of $\hat{v}_1,\ldots,\hat{v}_m$ we will also have that for each $1\leq i,k\leq m$.
\begin{eqnarray}
\Pi_m(\hat{v}_i\hat{v}_k^\ast)&=&\hat{V}^\ast \hat{v}_j\hat{v}_k^\ast\hat{V}\nonumber\\
&=&(\hat{V}^\ast \hat{v}_i) (\hat{V}^\ast \hat{v}_k)^\ast\nonumber\\
&=&\left(\sum_{j=1}^m \hat{e}_{j,m} \hat{v}_j^\ast \hat{v}_i\right)\left(\sum_{j=1}^m \hat{e}_{j,m} \hat{v}_j^\ast \hat{v}_k\right)^\ast\nonumber\\
&=&\left(\sum_{j=1}^m \hat{e}_{j,m} \delta_{j,i}\right)\left(\sum_{j=1}^m \hat{e}_{j,m} \delta_{j,k}\right)^\ast\nonumber\\
&=& \hat{e}_{i,m} \hat{e}_{k,m}^\ast
\label{matrix-units-generating-ident}
\end{eqnarray}
By \eqref{matrix-units-generating-ident} we have that $\{\hat{e}_{i,m} \hat{e}_{k,m}^\ast\}_{i,k=1}^m\subset \Pi_m(Z(P[\hat{v}_1|\hat{v}_m]))$, since $\Pi_m=\mathrm{Ad}[\hat{V}^\ast]\in CP(m,n)$ and $\mathbb{C}^{m\times m}=\mathrm{span}_{\mathbb{C}}\{\hat{e}_{i,m} \hat{e}_{k,m}^\ast\}_{i,k=1}^m$, we have that $\Pi_m$ is surjective. This completes proof.
\end{proof}

\begin{definition}
Given $v_1,\ldots,v_m\in \mathbb{C}^{n}\backslash\{\mathbf{0}\}$ with $n\geq 2m$, the matrix $U[v_1,v_m]$ whose existence is proved in \cref{orthonormalization-lemma} will be called a circular shift factor (CSF) for $v_1,\ldots,v_m$.
\end{definition}

\begin{lemma}
\label{local-pim-homomorphism}
Given $v_1,\ldots,v_m\in \mathbb{C}^n\backslash \{\mathbf{0}\}$ with $n\geq 2m$, there is an algebra homomorphism $\pi_m$ from $\mathbb{C}[U[v_1|v_m]]$ onto $\mathrm{Circ}(m)$.
\end{lemma}
\begin{proof}
By \cref{orthonormalization-lemma} we have that there is an orthornormal $m$-system $\hat{v}_1,\ldots,\hat{v}_m$ $\in$ $\mathbb{C}^{n}$ together with a CSF $U[v_1|v_m]=C[\hat{v}_1|\hat{v}_m]$. By \cref{orthonormalization-lemma} we have that 
$C[\hat{v}_1|\hat{v}_m]\in Z(P[\hat{v}_1|\hat{v}_m])$ with $P[\hat{v}_1|\hat{v}_m]$ determined by \eqref{p-cycle-proj-def}, this in turn implies that $\mathbb{C}[U[v_1|v_m]]$ $=$ $\mathbb{C}[C[\hat{v}_1|\hat{v}_m]]\subset Z(P[\hat{v}_1|\hat{v}_m])$. 

By \cref{projective-morphism} there is 
$\hat{V}\in \mathbb{C}^{n\times m}$ such that $\Pi_m=\mathrm{Ad}[\hat{V}^\ast]$ is $CP$ map from $\mathbb{C}^{n\times n}$ onto $\mathbb{C}^{m\times m}$ that preserves products in $Z(P[\hat{v}_1|\hat{v}_m])$.
Let us set $\pi_m=\Pi_m|_{\mathbb{C}[U[v_1|v_m]]}$. It is clear that $\pi_m\in CP(n,m)$. 

By \eqref{first-restriction-identity-Pim} and \eqref{matrix-units-generating-ident} we will have that.
\begin{eqnarray}
\pi_m(U[v_1|v_m])&=&\Pi_m(P[\hat{v}_1|\hat{v}_m]C[\hat{v}_1|\hat{v}_m])\nonumber\\
&=&\Pi_m \left(\hat{v}_1\hat{v}_m^\ast+\sum_{j=1}^{m-1}\hat{v}_{j+1}\hat{v}_j^\ast\right)\\
&=&\Pi_m( \hat{v}_1\hat{v}_m^\ast)+\sum_{j=1}^{m-1} \Pi_m( \hat{v}_{j+1}\hat{v}_j^\ast)\nonumber\\
&=&\hat{e}_{1,m}\hat{e}_{m,m}^\ast+\sum_{j=1}^{m-1} \hat{e}_{j+1,m}\hat{e}_{j,m}^\ast\nonumber\\
&=&\begin{bmatrix}
\mathbf{0}_{1\times(m-1)} & 1\\
\mathbf{1}_{m-1} & \mathbf{0}_{(m-1)\times 1}
\end{bmatrix}=C_m
\label{generating-identity-pim}
\end{eqnarray}

The identity \eqref{first-restriction-identity-Pim} also implies that.

\begin{equation}
\pi_m(\mathbf{1}_n)=\Pi_m(\mathbf{1}_n)=\Pi_m(P[\hat{v}_1|\hat{v}_m])=\mathbf{1}_m
\label{generating-unit-pim}
\end{equation}

Since $\Pi_m$ preserves products in $Z(P[\hat{v}_1|\hat{v}_m])$, for any two integers $j,k\geq 1$ we will have that,
\begin{eqnarray}
\pi_m(U[v_1|v_m]^j U[v_1|v_m]^k)&=&\Pi_m(U[v_1|v_m]^j U[v_1|v_m]^k)\nonumber\\
&=&\Pi_m(U[v_1|v_m]^j) \Pi_m(U[v_1|v_m]^k)\nonumber\\
&=&\pi_m(U[v_1|v_m]^j) \pi_m(U[v_1|v_m]^k)
\label{product-preserving-identity-pim-1}
\end{eqnarray}
and also that.
\begin{equation}
\pi_m(U[v_1|v_m]^j)=\Pi_m(U[v_1|v_m]^j)=\Pi_m(U[v_1|v_m])^j=\pi_m(U[v_1|v_m])^j=C_m^j
\label{product-preserving-identity-pim-2}
\end{equation}
By \eqref{generating-unit-pim}, \eqref{product-preserving-identity-pim-1} and \eqref{product-preserving-identity-pim-2}, we will have that the map $\pi_m\in CP(n,m)$ determines an algebra homomorphism from $\mathbb{C}[U[v_1|v_m]]$ onto $\mathrm{Circ}(k)$.
\end{proof}

\begin{definition}
Given $v_1,\ldots,v_m\in \mathbb{C}^{n}\backslash\{\mathbf{0}\}$ with $n\geq 2m$, with corresponding CSF $U[v_1,v_m]\in \mathbb{U}(n)$. The algebra homomorphism $\pi_m$ whose existence is warranteed by \cref{local-pim-homomorphism}, will be called a Circulant representation (CR) for $\mathbb{C}[U[v_1|v_m]]$
\end{definition}

\begin{theorem}
\label{solvability-TC-diagram}
Given $v_1,\ldots,v_m\in \mathbb{C}^n\backslash \{\mathbf{0}\}$ with $n\geq 2m$, there is a projection $P\in \mathbb{C}^{m\times m}$ together two maps $\varphi\in CP(n,n)$ and $\Phi\in CP(m,n)$, such that the following diagram commutes,
\begin{equation}
\xymatrix{
&   & \mathrm{Circ}(m) \ar@{-->}[d]^\Phi\\
& \mathbb{C}[U[v_1|v_m]] \ar[ru]^{\pi_m} \ar@{-->}[r]_\varphi & Z(P)
}
\label{TC-diagram}
\end{equation}
where $\pi_m$ is a CR of $\mathbb{C}[U[v_1|v_m]]$. Moreover, $\Phi$ preserves products on $\mathbb{C}^{m\times m}$ and 
$PU[v_1|v_m]=U[v_1|v_m]P$.
\end{theorem}
\begin{proof}
By \ref{orthonormalization-lemma} there is an orthonormal $m$-system $\hat{v}_1,\ldots,\hat{v}_m\in \mathbb{C}^n$ together with projection $P[\hat{v}_1|\hat{v}_m]$ such that $P[\hat{v}_1|\hat{v}_m]\hat{v}_j=\hat{v}_j$ for each $1\leq j\leq m$. As a consequence of the argument implemented in the proof of \cref{local-pim-homomorphism}, we have that by \cref{projective-morphism}, there is a matrix $\hat{V}\in \mathbb{C}^{n\times m}$ such that $\pi_m=\mathrm{Ad}[\hat{V}^\ast]|_{\mathbb{C}[U[v_1|v_m]]}$ and $P[\hat{v}_1|\hat{v}_m]=\hat{V}\hat{V}^\ast$.

Let us set. 
\begin{equation}
\left\{
\begin{array}{l}
P=P[\hat{v}_1|\hat{v}_m]\\
\varphi=\mathrm{Ad}[P]\\
\Phi=\mathrm{Ad}[\hat{V}]
\end{array}
\right.
\label{TC-objects-def}
\end{equation}
It is clear that $\varphi\in CP(n,n)$ and $\Phi\in CP(m,n)$. Given $X\in \mathbb{C}^{n\times n}$, we will have that.
\begin{equation}
\varphi(X)=P[\hat{v}_1|\hat{v}_m]XP[\hat{v}_1|\hat{v}_m]=\hat{V}\hat{V}^\ast X\hat{V}\hat{V}^\ast=\hat{V}\Pi_m( X)\hat{V}^\ast=\Phi(\Pi_m(X))
\label{composition-identity-1}
\end{equation}

Since $\hat{V}^\ast \hat{V}=\mathbf{1}_m$ we will have that for any two $X,Y\in \mathbb{C}^{m\times m}$.
\begin{equation}
\Phi(XY)=\hat{V}XY\hat{V}^\ast=\hat{V}X\hat{V}^\ast\hat{V}Y\hat{V}^\ast=\Phi(X)\Phi(Y)
\label{prod-preserver-Phi}
\end{equation}

Since $P=P[\hat{v}_1|\hat{v}_m]$ is a projection, for any $X\in \mathbb{C}^{n\times n}$, we will have that $P\varphi(X)=P^2XP=PXP^2=\varphi(X)P$. This imples that $\varphi(\mathbb{C}^{n\times n})\subseteq Z(P)$.

By \eqref{lema-tool-2} we have that $PU[v_1|v_m]=P[\hat{v}_1|\hat{v}_m]C[\hat{v}_1|\hat{v}_m]=C[\hat{v}_1|\hat{v}_m]P[\hat{v}_1|\hat{v}_m]=U[{v}_1|{v}_m]P$.

By \eqref{composition-identity-1} we will have that $\varphi$ has a representation of the form.
\begin{equation}
\varphi=\Phi\circ \Pi_m
\label{composition-identity-2}
\end{equation}
By \eqref{composition-identity-2} we will have that. 
\begin{equation}
\varphi|_{\mathbb{C}[U[v_1|v_m]]}=\Phi\circ \Pi_m|_{\mathbb{C}[U[v_1|v_m]]}=\Phi\circ \pi_m
\label{composition-identity-3}
\end{equation}
This completes the proof.
\end{proof}

Given a discrete-time dynamical system $(\Sigma,\{\Theta_t\})$ with $\Sigma\subseteq \mathbb{C}^{n\times n}$, and a $m$-system of (history) vectors $\mathbb{H}[\Sigma,m]=\{v_1,\ldots,v_m\}$ for $(\Sigma,\{\Theta_t\})$, the matrix $\hat{\mathbb{V}}\in \mathbb{C}^{n\times m}$ defined by the formula \eqref{history_matrix_svd-4} for $\{v_1,\ldots,v_m\}$, will be called the orthonormal history factor (OHF) of $\mathbb{H}[\Sigma,m]$.

\begin{theorem}
\label{main-topological-control-result}
Given $\varepsilon>0$, an integer $T>0$, and a discrete-time $T$-periodic dynamical system $(\Sigma,\{\Theta_t\})$ with $\Sigma\subseteq \mathbb{C}^{n\times n}$, then $(\Sigma,\{\Theta_t\},\mathbb{H}[\Sigma,m])$ is $(\mathbb{S}^1,T,\varepsilon)$-controlled, for every $\varepsilon>0$ and each $m\in\mathbb{Z}$ such that $2T\leq 2m\leq n$.
\end{theorem}
\begin{proof}
Given $\varepsilon>0$, and any vector history $\mathbb{H}[\Sigma,m]=\{v_1,\ldots,v_m\}\subset \mathbb{C}^n$ for $(\Sigma,\{\Theta_t\})$, since $2T\leq 2m\leq n$, we can apply \cref{orthonormalization-lemma} to compute an orthonormal $T$-system $\{\hat{v}_1,\ldots,\hat{v}_T\}\in\mathbb{C}^n$, scalars $\kappa,\rho$, a CSF $U[v_1|v_T]\in\mathbb{C}^{n\times n}$ and two projections $\hat{K},\hat{T}\in \mathbb{C}^{n\times n}$, that satisfy \eqref{cyclic-constraints}.

By \cref{solvability-TC-diagram} we will have that there are a projection $P\in \mathbb{C}^{n\times n}$, $\varphi\in CP(n,n)$ and a product preserver $\Phi\in CP(T,n)$ such that $\varphi=\Phi\circ \pi_T$ and $\varphi(\mathbb{C}[U[v_1|v_T]])\subseteq Z(P)$.

Since $\pi_T$ and $\Phi$ are linear product preservers in $\mathbb{C}[U[v_1|v_T]]$ and $\mathbb{C}^{T\times T}$, respectively, we will have that for any $p\in \mathbb{C}[z]$.

\begin{equation}
\varphi(p(U[v_1|v_T]))=\Phi(\pi_m(p(U[v_1|v_T])))=p(\Phi(\pi_m(U[v_1|v_T])))=p(\varphi(U[v_1|v_T]))
\label{polynomial-preserver-identity}
\end{equation}

By \cref{solvability-TC-diagram} we will have that,
\begin{equation}
\|\varphi(X)-\Phi\circ\pi_T(p(U[v_1|v_T])))\|=0<\varepsilon
\label{algebraic-rep-identity}
\end{equation}
for each $X\in \mathbb{C}[U[v_1|v_T]]$.

By \eqref{TC-objects-def} we have that $P=P[\hat{v}_1|\hat{v}_T]$, with $P[\hat{v}_1|\hat{v}_T]$ defined by equation \eqref{p-cycle-proj-def}, by \cref{First_Projective_lemma} we will have that.

\begin{equation}
P\hat{v}_1=P[\hat{v}_1|\hat{v}_T]\hat{v}_1=\hat{v}_1
\label{projective-preserved-vector}
\end{equation}

Let $p_k(z)=\frac{\kappa}{\rho} z^k$, $0\leq k\leq m-1$, this impies that $p_k\in \mathbb{C}[z]$. By \cref{orthonormalization-lemma} and by \eqref{projective-preserved-vector}, for each $0\leq k\leq m-1$ we will have that.
\begin{eqnarray}
\hat{K}\varphi(p_k(U[v_1|v_T]))\hat{T}v_1&=&\hat{K}\frac{\kappa}{\rho}U[v_1|v_T]^k P\rho\hat{v}_1\nonumber\\
&=& \hat{K}\kappa U[v_1|v_T]^k P\hat{v}_1\nonumber\\
&=& \hat{K}\kappa U[v_1|v_T]^k\hat{v}_1\nonumber\\
&=& \hat{K}\kappa \hat{v}_{1+k}=v_{k+1}=\Theta_{k}(v_1)
\label{algebraic-evolution-identity}
\end{eqnarray}

By \eqref{algebraic-rep-identity} and \eqref{algebraic-evolution-identity}, we will have that $(\Sigma,\{\Theta_t\},\mathbb{H}[\Sigma,m])$ is $(\mathbb{S}^1,T,\varepsilon)$-controlled by $(\mathbb{S}^1,U[v_1|v_T],\hat{K},\hat{T},\varphi,\{p_t\})$.
\end{proof}

\begin{theorem}
\label{main_algebraic_control_result}
Given $\varepsilon>0$, an integer $T>0$, and a discrete-time $T$-periodic dynamical system $(\Sigma,\{\Theta_t\})$ with $\Sigma\subseteq \mathbb{C}^{n\times n}$, then $(\Sigma,\{\Theta_t\},\mathbb{H}[\Sigma,m])$ is $(\mathbb{Z}/T,\varepsilon)$-controlled, for every $\varepsilon>0$ and each $m\in\mathbb{Z}$ such that $2T\leq 2m\leq n$.
\end{theorem}
\begin{proof}
Given $\varepsilon>0$, by \cref{main-topological-control-result}, we will have that $(\Sigma,\{\Theta_t\},\mathbb{H}[\tilde{\Sigma},m])$ is $(\mathbb{S}^1$ $,$ $T$ $,$ $\varepsilon)$-controlled by some control $(\mathbb{S}^1,U[v_1|v_T],\hat{K},\hat{T},\varphi,\{p_t\})$.

We have that \cref{main-topological-control-result} also implies that there is an algebra homomorphism 
$\Phi:\mathbb{C}^{m\times m}\to \mathbb{C}^{n\times n}$ such that $\|\varphi(X)-\Phi\circ \pi_T(X)\|\leq \varepsilon$, for each $X\in \mathbb{C}[U[v_1|v_T]]$.

Since $U[v_1|v_T]\in \mathbb{U}(n)$, we will have that $C_{T}=\pi_m(U[v_1|v_T])\in \mathbb{U}(m)$ and also that.
\[
C_T^T=\pi_T(U[v_1|v_m])^T=\pi_T(U[v_1|v_m]^T)=\pi_T(\mathbf{1}_n)=\mathbf{1}_m
\]

By universality of $\mathbb{C}[\mathbb{Z}/T]$ there is a representation $\rho_T:\mathbb{C}[\mathbb{Z}/T]\to \mathbb{U}(T)$, determined by the assignment $\hat{1}\mapsto C_T$ for $\hat{1}\in \mathbb{Z}/T$.

Applying universality of $\mathbb{C}[\mathbb{Z}/T]$ to $\pi_T(\mathbb{C}[U[v_1|v_T]])$ we have that, 
\[
\pi_T(p_t(U[v_1|v_T]))=p_t(\pi_T(U[v_1|v_T]))=p_t(C_T)=p_t(\rho_T(\hat{1}))=\rho_T(p_t(\hat{1}))
\]
for each $0\leq t\leq m-1$. This completes the proof.
\end{proof}

\begin{theorem}
\label{top_geo_control}
Given $\delta,\varepsilon>0$, every $(T,\delta,\varepsilon)$-almost-periodic dynamical system $(\Sigma,\{\Theta_t\})$ with $\Sigma\subseteq \mathbb{C}^n$ is ($\mathbb{S},T,\varepsilon$)-controlled, whenever $2T\leq 2m\leq n$.
\end{theorem}
\begin{proof}
Let $x\in \Sigma$. Since $(\Sigma,\{\Theta_t\})$ is $(T,\delta,\varepsilon)$-almost-periodic, we will have that there is a discrete-time $T$-periodic dynamical system $(\tilde{\Sigma},\tilde{\Theta}_t)$ such that there is $\tilde{x}\in \tilde{\Sigma}$ such that $\|x-\tilde{x}\|\leq \delta$, and $\|\Theta_t(x)-\tilde{\Theta}_t(\tilde{x})\|_2\leq \varepsilon$ for each $t\in \mathbb{Z}_0^+$ and every $(x,\tilde{x})\in \Sigma\times \tilde{\Sigma}$.

By \cref{main-topological-control-result} we will have that $(\tilde{\Sigma},\{\tilde{\Theta}_t\},\mathbb{H}[\tilde{\Sigma},T])$ is $(\mathbb{S}^1,T,0)$-controlled by some control $(\mathbb{S}^1,U[v_1|v_T],\hat{K},\hat{T},\varphi,\{p_t\})$, for some $T$-system of history vectors $\mathbb{H}[\tilde{\Sigma},T]$ $=$ $\{v_1$ $,$ $\ldots$ $,$ $v_T\}$ with $v_1=\tilde{x}$, this implies that for each $0\leq t\leq T$.
\begin{eqnarray}
\|\Theta_t(x)-\hat{K}\varphi(p_t(Z))\hat{T}\tilde{x}\|_2&\leq& \|\Theta_t(x)-\tilde{\Theta}_t(\tilde{x})\|_2+\|\tilde{\Theta}_t(\tilde{x})-\hat{K}\varphi(p_t(Z))\hat{T}\tilde{x}\|_2\nonumber\\
&\leq&\|\Theta_t(x)-\tilde{\Theta}_t(\tilde{x})\|_2\leq \varepsilon
\label{normed_control_constraints}
\end{eqnarray}
By \eqref{normed_control_constraints} we have that $(\Sigma,\{\Theta_t\})$ is $(\mathbb{S}^1,T,\varepsilon)$-controlled by $(\mathbb{S}^1$ $,$ $U[v_1|v_T],\hat{K}$ $,$ $\hat{T}$ $,$ $\varphi,\{p_t\})$. This completes the proof.
\end{proof}

\begin{theorem}
\label{top_alg_control}
Given $\delta,\varepsilon>0$, every $(T,\delta,\varepsilon)$-almost-periodic dynamical system $(\Sigma,\{\Theta_t\})$ with $\Sigma\subseteq \mathbb{C}^n$ is ($\mathbb{Z}/T,\varepsilon$)-controlled, whenever $2T\leq n$.
\end{theorem}
\begin{proof}
Let $x\in \Sigma$. Since $(\Sigma,\{\Theta_t\})$ is $(T,\delta,\varepsilon)$-almost-periodic, we will have that there is a discrete-time $T$-periodic dynamical system $(\tilde{\Sigma},\tilde{\Theta}_t)$ such that there is $\tilde{x}\in \tilde{\Sigma}$ such that $\|x-\tilde{x}\|\leq \delta$, and $\|\Theta_t(x)-\tilde{\Theta}_t(\tilde{x})\|_2\leq \varepsilon$ for each $t\in \mathbb{Z}_0^+$ and every $(x,\tilde{x})\in \Sigma\times \tilde{\Sigma}$.

By \cref{main_algebraic_control_result} we will have that $(\tilde{\Sigma},\{\tilde{\Theta}_t\},\mathbb{H}[\Sigma,T])$ is $(\mathbb{Z}/T,0)$-controlled, this implies that for some $T$-system of history vectors $\mathbb{H}[\tilde{\Sigma},T]$ $=$ $\{v_1$ $,$ $\ldots$ $,$ $v_T\}$ with $v_1=\tilde{x}$, there is a unitary representation $\rho_T:\mathbb{C}[\mathbb{Z}/T]\to \mathbb{C}^{T\times T}$, an algebra homomorphism $\varphi:\mathbb{C}^{T\times T}\to \mathbb{C}^{n\times n}$, a family of functions $\{f_0,\ldots,f_{T-1}\}$, and two projections $\hat{K},\hat{T}\in \mathbb{C}^{n\times n}$ such that for $\hat{1}\in\mathbb{Z}/T$ and for each $t\geq 0$.
\begin{eqnarray}
\|\Theta_t(x)-\hat{K}\varphi\circ\rho_T(f_k(\hat{1}))\hat{T}\tilde{x}\|_2&\leq& \|\Theta_t(x)-\tilde{\Theta}_t(\tilde{x})\|_2\nonumber\\
&&+\|\tilde{\Theta}_t(\tilde{x})-\hat{K}\varphi\circ\rho_T(f_k(\hat{1}))\hat{T}\tilde{x}\|_2\nonumber\\
&\leq&\|\Theta_t(x)-\tilde{\Theta}_t(\tilde{x})\|_2\leq \varepsilon
\label{normed_control_algebraic_constraints}
\end{eqnarray}
By \eqref{normed_control_constraints} we have that $(\Sigma,\{\Theta_t\})$ is $(\mathbb{Z}/T,\varepsilon)$-controlled. This completes the proof.
\end{proof}

Given $\varepsilon>0$, from the definition in \eqref{eq:loc_bil_rep_sigma} of SCL-ROM approximants of a discrete-time system $(\Sigma,\{\Theta_t\})$ with $\Sigma\subseteq \mathbb{C}^n$, we have that by \cref{orthonormalization-lemma} and \cref{top_geo_control}, every discrete-time system $(\Sigma,\{\Theta_t\})$ that is $(\mathbb{S}^1,\varepsilon)$-controlled by some control $(\mathbb{S}^1,Z$ $,$ $K$ $,$ $T$ $,$ $\varphi,\{f_t\})$, has a SCL-ROM approximant of the form \eqref{eq:loc_bil_rep_sigma} with $F_t=\varphi(f_t(Z))$ for $0\leq t\leq T-1$, we say that such approximant is SCL-ROM is determined by the control $(\mathbb{S}^1,Z,K$ $,$ $T$ $,$ $\varphi$ $,$ $\{f_t\})$. We can think of the evolution laws $\hat{v}_{t+1}=\hat{F}_t\hat{v}_t$ for $0\leq t\leq 1$ in SCL-ROM the decomposition \eqref{eq:loc_bil_rep_sigma}, as an embedding of the original system in a manifold, where the evolution of the embedded system is controlled by a unitary matrix determined the CSF of some vector history for the original system.

\section{Algorithm}
\label{sec:alg}

The techniques and compuations used to prove the previous results, can be implemented to derive a prototypical algorithm described by the following block diagram.

\vspace*{1.5pc}

\tikzstyle{block} = [draw, fill=white, rectangle, 
    minimum height=3em, minimum width=6em]
\tikzstyle{sum} = [draw, fill=white, circle, node distance=1cm]
\tikzstyle{input} = [coordinate]
\tikzstyle{output} = [coordinate]
\tikzstyle{pinstyle} = [pin edge={to-,thin,black}]

\begin{equation}
\label{main_block_diagram}
\begin{tikzpicture}[auto, node distance=2cm,>=latex']
    \node [input, name=input] {};
    \node [block, right of=input] (delay) {$Z^{-1}$};
    \node [block, right of=delay, node distance=3cm] (K) {$\hat{K}$};

    \draw [->] (delay) -- node[name=u] {$\hat{x}_t$} (K);
    
    \node [output, right of=K] (output) {};
    \node [block, below of=delay, pin={[pinstyle]below: $\mathbb{C}[\mathbb{Z}/m]$},
            node distance=2cm] (feedback) {$F_t$};

    \draw [->] (input) -- node {$\hat{x}_{t+1}$} (delay);
    \draw [->] (K) -- node [name=y] {$x_t$}(output);
    \draw [-] (feedback) -| node[name=input] {$$} 
        node [near end] {$$} (input);
    \draw [->] (u) |- (feedback);    
\end{tikzpicture}
\end{equation}

The proofs of \cref{orthonormalization-lemma} and \cref{main-topological-control-result} provide a computational procedure that is sketched in \cref{alg:main_alg}, and can be used to compute the elements ${F_t}_{t=0}^{T-1}$ in diagram \eqref{main_block_diagram}, where each matrix $F_t$ has a reresentation $F_t=\hat{K}\hat{H}_t\hat{T}$, for some matrices $\hat{K}$, $\hat{H}_t$, $\hat{T}$ to be determined by \cref{alg:main_alg}.

\begin{algorithm}[h!]
\caption{Data-driven matrix approximation algorithm}
\label{alg:main_alg}
\begin{algorithmic}
\STATE{{\bf Data:}\:\:\:{\sc Real number $\varepsilon>0$, State History $\mathbb{H}[\Sigma,m]$:} $\{x_t\}_{0\leq t\leq m\leq T}$, $T\in \mathbb{Z}^+$}
\STATE{{\bf Result:}\:\:\: {{\sc Approximate matrix realizations:} $(\hat{K},\hat{H}_t,\hat{T})\in \mathbb{C}^{n\times n}\times \mathbb{C}^{n\times n}\times \mathbb{C}^{n\times n}$ of $(\tilde{\Sigma},\{\Theta_t\})$}}
\begin{enumerate}
\STATE{Compute state/output sampled-data history $\{v_t\}_{0\leq t\leq T}$ of $(\Sigma,\{\Theta_t\})$\;}
\STATE{Compute the SVD $\mathbb{V}S\mathbb{W}=[v_1~\cdots~v_m]$ of $\{v_t\}_{0\leq t\leq m\leq T}$\;}
\STATE{Compute the OHF $\hat{\mathbb{V}}=[\hat{v}_1~\cdots~\hat{v}_m]$ for $\{v_t\}_{0\leq t\leq m\leq  T}$\;}
\STATE{Set $\hat{K}=\mathbb{V}\mathbb{V}^\ast$\;}
\STATE{Set $\hat{T}=\hat{v}_1\hat{v}_1^\ast$\;}
\STATE{For $0\leq t\leq T-1$:}
\begin{enumerate}
\STATE{Compute $p_t\in \mathbb{C}[z]$ such that:\;}
\begin{enumerate}
\STATE{$\|\hat{K}\hat{\mathbb{V}}p_t(C_m^t)\hat{\mathbb{V}}^\ast\hat{T}v_1-\Theta_t[v_1]\|\leq \varepsilon$\;}
\end{enumerate}
\STATE{Set $\hat{H}_t=\hat{\mathbb{V}}p_t(C_m^t)\hat{\mathbb{V}}^\ast$\;}
\end{enumerate}
\end{enumerate}
\RETURN $\{\hat{K},\hat{H}_t,\hat{T}(t)\}_{0\leq t\leq T}$
\end{algorithmic}
\end{algorithm}

\section{Experimental results}
\label{sec:experiments}


\subsection{Materials and Methods}

In order to solve the diagram \eqref{main_block_diagram}, a prototypical GNU Octave code that implements some of the core computations on which the proofs of \cref{orthonormalization-lemma} and \cref{main-topological-control-result} are based, has been developed as part of this project, using GNU Octave 4.4.1 on a five node Ubuntu Linux Beowulf Cluster, at the Scientific Computing Innovation Center of UNAH-CU.

\subsection{Numerical Experiments}
We will consider three numerical experiments in this section. In \S\ref{waves} we will simulate 1D waves, in \S\ref{damped_waves} we will simulate damped waves on planar material sections, and in \S\ref{vortex} we will simulate vorticity transport.

In \S\ref{waves} the periodicity appears naturally, while in \S\ref{damped_waves} and \S\ref{vortex}, we will think of the model as a movie, that is being streamed more than once. Each example will be used to illustrate the potential applications of topological control to system identification, and for extraction of (almost) periodic patterns from sampled-data discrete-time industrial systems and plants.

\subsubsection{1D Waves} \label{waves} As a first application, let us consider a wave equation under Dirichlet boundary conditions of
the form:
\begin{equation}
\left\{
\begin{array}{l}
\partial_t^2 w-c^2\partial_tx^2 w=0,\\
w(0,t)=w(L,t)=0,\\
w(x,0)=w_0(x),\\
\partial_t w(x,0)=0
\end{array}
\right.
\label{wave_DS}
\end{equation}
for some suitable data of $L,c,w_0$.

The simulation is computed using second order finite difference method combined with second order Crank-Nicolson method. The computation is performed using the following commands.

\begin{verbatim}
>> [x,t,data_wave,Cx,Sx]=CL_ROM_WaveDS([0 1],10);
-------------------------------------------------------------
                     Running simulation:
-------------------------------------------------------------
-------------------------------------------------------------
 Computing circular matrix representations in C[U[v1|vm]]:
-------------------------------------------------------------
\end{verbatim}

The graphical output is shown in \cref{wave_DS_output}.

\begin{figure}[h!]
\begin{center}\vspace{1cm}
\includegraphics[width=0.7\linewidth]{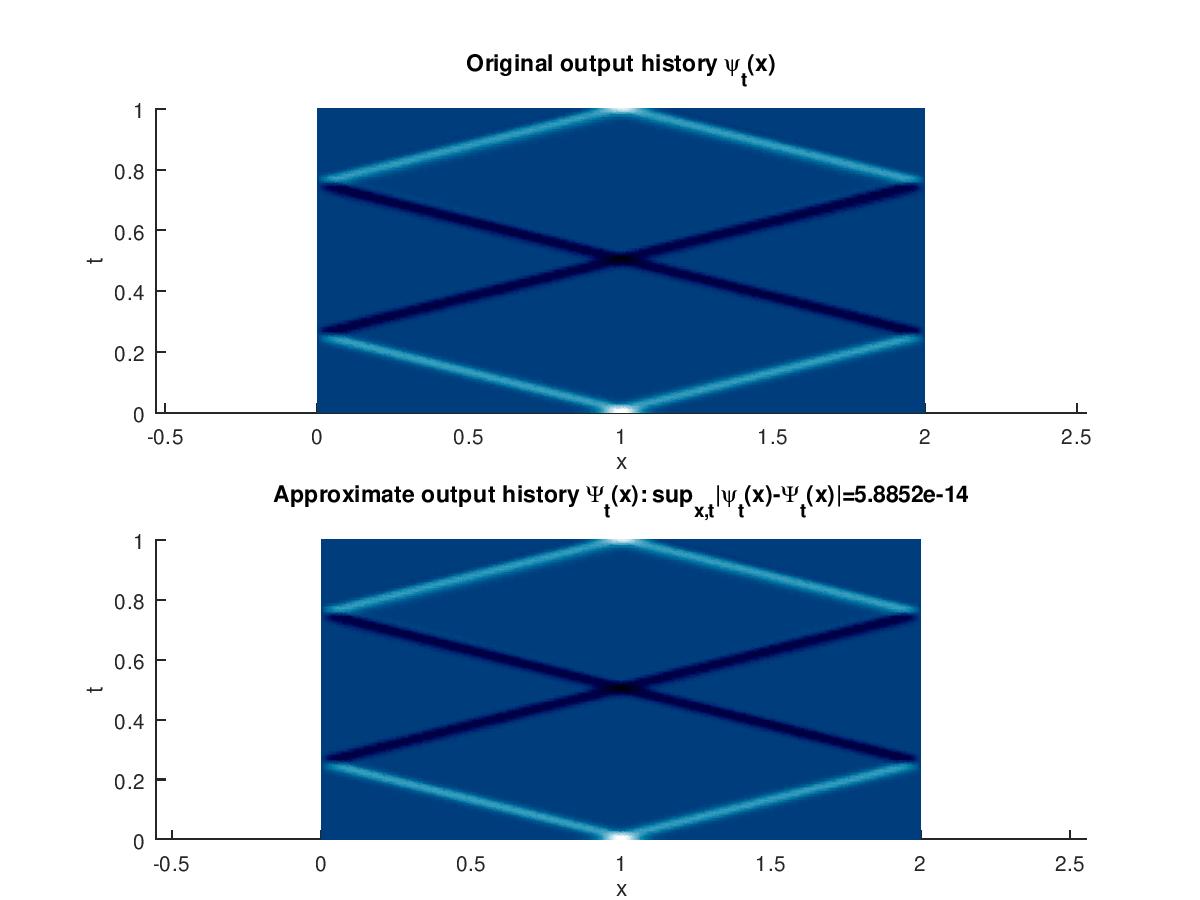}
\caption{$\Sigma$ and $\hat{\Sigma}$ output histories}
\label{wave_DS_output}
\end{center}\vspace{1cm}
\end{figure}

\subsubsection{Cantilever Elastic Plate} \label{damped_waves}
As a second application of \cref{alg:main_alg}, let us consider a computational mechanics problem consisting on the description of the deformation a damped aluminium Cantilever Lam\'e beam model under planar displacement hypotheses, whose deformation displacement vector $v$ is described by a Navier dynamical system of the form:
\begin{equation}
\left\{
\begin{array}{l}
\rho \partial_t^2 v-N(\lambda,\mu)v=u(t),\\
BIC(v)=0,
\end{array}
\right.
\label{Navier_DS}
\end{equation}
where $N(\lambda,\mu)$ is the Navier operator defined by the expression,
\[
N(\lambda,\mu)=(\lambda+\mu)\nabla\nabla\cdot+\mu\nabla^2
\]
where $\lambda,\mu$ are the Lam\'e's coefficients for generic aluminium, and where $BIC(v)=0$ is some system of equations that determines suitable boundary and initial conditions for Cantilever Lam\'e beam deflection. We will have that the input $u(t)$ is determined by the expression $u(t)=c(t)\rho\partial_t v$ for some smooth time dependent coefficient $c(t)$.

It is important to consider that the dimensions of the beam model have been normalized, and that relative deformation displacement scale is exaggerated for visualization purposes of the corresponding simulation. 

In order to create the data corresponding to the beam deformation, we use an Octave m-file function that computes a second order finite difference approximation of the Navier dynamical system \eqref{Navier_DS} for sample sizes $50$, $30$, $10$ and $5$, as follows.

\begin{verbatim}
>> [Bx,By,data_x,data_y,Yx,Yy]=CL_ROM_BeamDS(1,50,40,-5e9,-10,1);
>> [Bx,By,data_x,data_y,Yx,Yy]=CL_ROM_BeamDS(1,30,40,-5e9,-10,2);
>> [Bx,By,data_x,data_y,Yx,Yy]=CL_ROM_BeamDS(1,10,40,-5e9,-10,3);
>> [Bx,By,data_x,data_y,Yx,Yy]=CL_ROM_BeamDS(1,5,40,-5e9,-10,4);
-------------------------------------------------------------
                     Running simulation:
-------------------------------------------------------------
properties =

  scalar structure containing the fields:

    Emod =  73100000000

properties =

  scalar structure containing the fields:

    Emod =  73100000000
    Nu =  0.33000
    Lx =  4
    Ly =  0.40000
    T =  40
    M =  600
    N =  36
    c2 =  0.35971
    delta = -5000000000
    m =  5

Elapsed time is 298.816 seconds.
-------------------------------------------------------------
 Computing circular matrix representations in C[U[v1|vm]]:
-------------------------------------------------------------
Elapsed time is 3.51714 seconds.
-------------------------------------------------------------
 Verifying circular mimetic constraints for C[U[v1|vm]]:
-------------------------------------------------------------
Verification passed...
max{||Kx Rx Usx^k Tx Ux0-Oxk|| | 1<=k<=m} = 2.7082e-27 <= eps
max{||Ky Ry Usy^k Ty Uy0-Oyk|| | 1<=k<=m} = 7.4542e-13 <= eps
-------------------------------------------------------------
For m = 121
For n = 13357
For eps = 7.4542e-13
Elapsed time is 20.7964 seconds.
-------------------------------------------------------------
\end{verbatim}

This produces the original output history shown in Figure \cref{original_output}.

\begin{figure}[h!]
\begin{center}\vspace{1cm}
\includegraphics[width=0.56\linewidth]{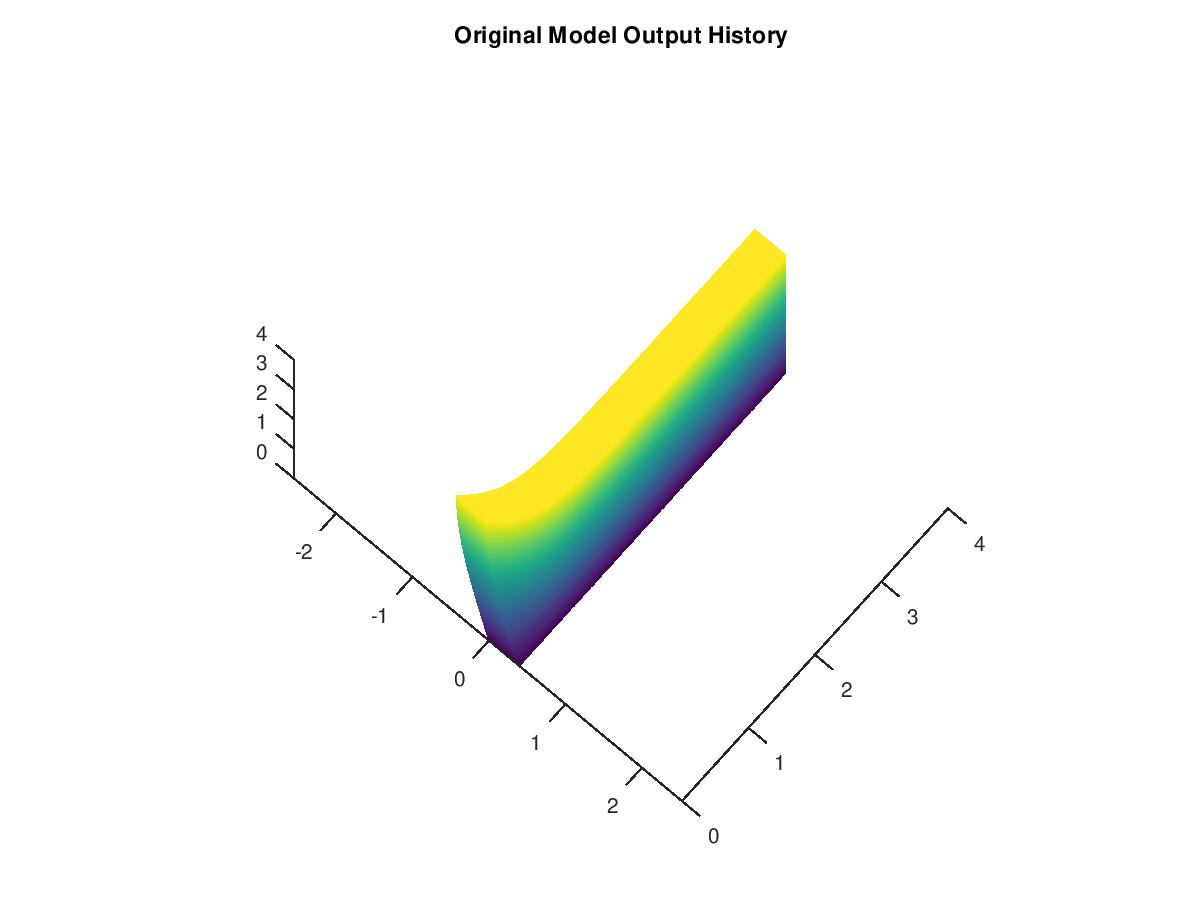}
\caption{Original model $\Sigma$ output history}
\label{original_output}
\end{center}\vspace{1cm}
\end{figure}

The output histories of the corresponding SCL-ROM approximants are shown in Figure \ref{SCL-ROM_outputs}.

\begin{figure}[h!]
\begin{center}\vspace{1cm}
\includegraphics[width=0.56\linewidth]{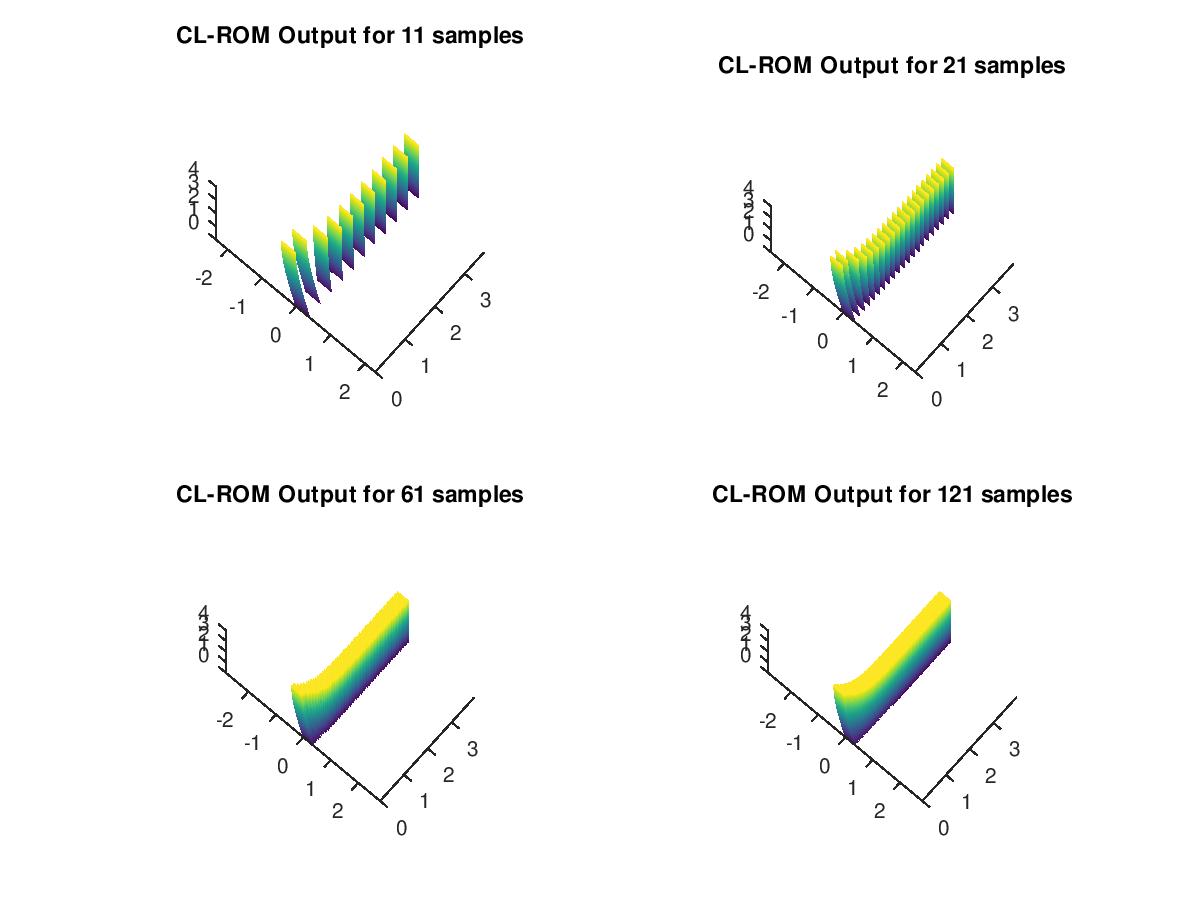}
\caption{ Output histories for several SCL-ROM approximants $\tilde{\Sigma}$ of $\Sigma$}
\label{SCL-ROM_outputs}
\end{center}\vspace{1cm}
\end{figure}

\subsubsection{Planar vorticity transport} \label{vortex}
As a third application of local $\mathbb{S}^1$-control, let us consider the vorticity transport PDE system of the form.
\begin{equation*}
\left\{
\begin{array}{l}
\partial_t \omega=-u\partial_x \omega-v\partial_y \omega+\frac{1}{Re}\Delta \omega\\
u=\partial_y \psi,\\
v=-\partial_x \psi\\
\Delta \psi=-\omega\\
BIC(\omega,\psi,u,v)=0
\end{array}
\right.
\end{equation*}
For Reynolds number $Re=400$ and suitable boundary and initial conditions represented by the system of algebraic differential equaitons $BIC$.

The simulation is computed using second order finite difference method combined with fourth order Runge-Kuta method. The computation is performed using the following commands.
\begin{verbatim}
>> [x,y,data,Cx,Sx]=CL_ROM_VortexDS(1,10,[0,0.15],[100,750],[1,3]);
-------------------------------------------------------------
                     Running simulation:
-------------------------------------------------------------
-------------------------------------------------------------
 Computing circular matrix representations in C[U[v1|vm]]:
-------------------------------------------------------------
Elapsed time is 2.08992 seconds.
-------------------------------------------------------------
Verifying circular mimetic constraints for C[U[v1|vm]]:
-------------------------------------------------------------
max{||Kx Rx Ux^k Tx X0-Yk|| | 1<=k<=m} = 8.2096e-13 <= eps
-------------------------------------------------------------
For m = 76
For n = 34501
For eps = 8.2096e-13
Elapsed time is 111.598 seconds.
-------------------------------------------------------------
\end{verbatim}

The graphical output is shown in \cref{vortex_DS_output}.

\begin{figure}[h!]
\begin{center}\vspace{1cm}
\includegraphics[width=0.77\linewidth]{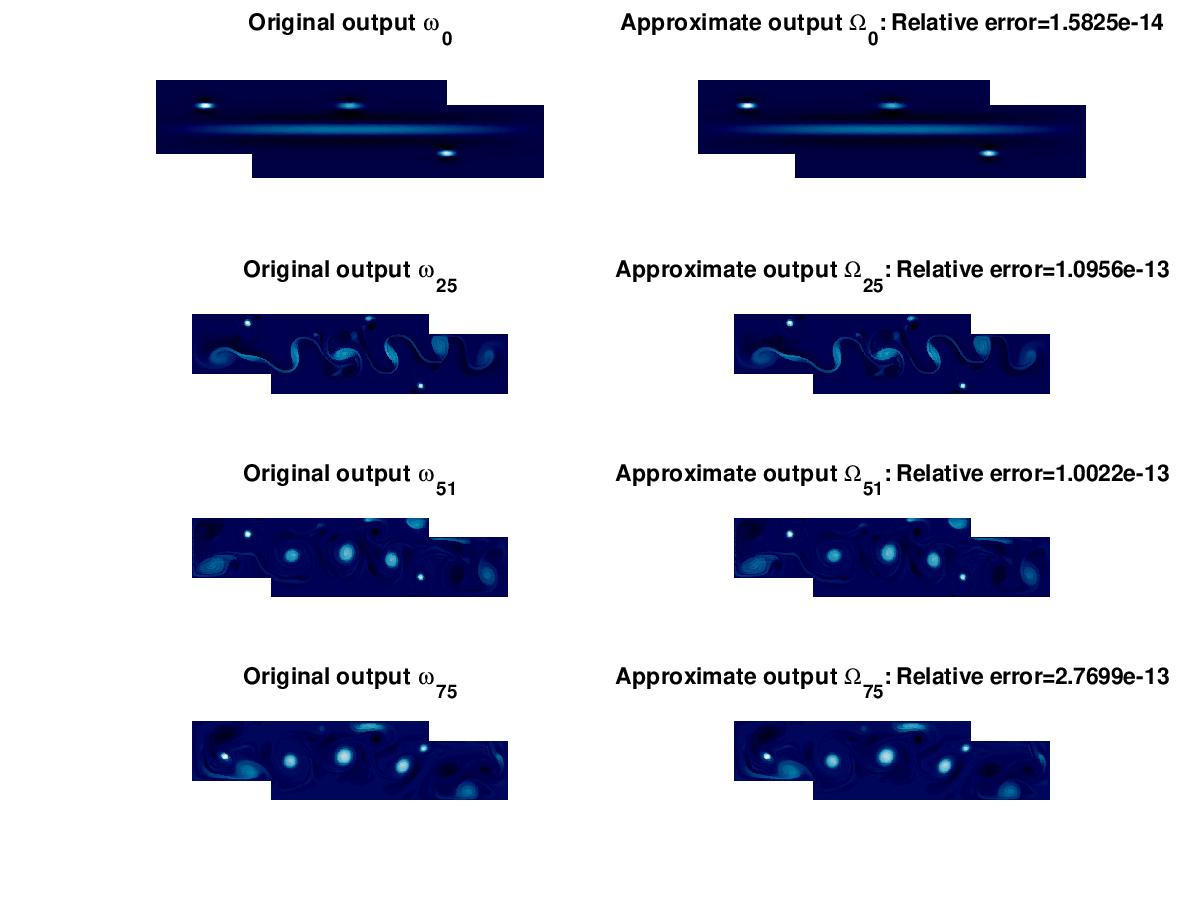}
\caption{$\Sigma$ and $\hat{\Sigma}$ output histories}
\label{vortex_DS_output}
\end{center}\vspace{1cm}
\end{figure}

\section{Discussion}

From a topological perspective, the notion of topological control that we propose in this document can be seen as an extension of the Torus Trick presented by R. Kirby in \cite{Kirby} to algebraic matrix sets in the sense of \cite{Vides_algebraic}. This extension and the corresponding 
matrix computations, were partially inspired by some questions raised by M. H. Freedman along the lines of \cite{Freedman}.

In order to perform the previously mentioned computations, we start embedding the vector history of a given discrete-time system under study into a 
manifold, where we can then we use elementary tools from matrix analysis and representation theory, to compute matrix analogies of the surgical cuts corresponding to Kirby's torus trick, these matrix surgical cuts have a direct effect on the spectrum of the CSF corresponding to the vector history of the corresponding embedded system.

Once we perform the previously mentioned surgical cuts on the spectrum of the unitary matrices that model the dynamical behaviour of the embedded system under study, the computation of the transition matrices of the embedded system can be easily and efficiently computed in terms of the topologically pre-processed matrices.

Another interesting effect of the aformentioned matrix surgical procedures, consists on the reduction of the group action that determines the global behaviour of the system, to a finite group action that can be efficiently computed using \cref{alg:main_alg}, without aditional computational cost due to the additional liftings in the original state (matrix realization) space, whose large dimension can make the standard lifting impossible to compute. This approach was inspired by the work of M. Rieffel in \cite{Rieffel}, and will be the subject of further study.

\subsection{Forthcoming Research}

We will further explore the numerical solvability of \eqref{main_matrix_equation} together with some additional contraints.  

We will improve the computational implementation of the prototypical \cref{alg:main_alg}, extending the topological control techniques presnted in this document to higher dimensional compact manifolds like $\mathbb{S}^1\times \mathbb{S}^1$, $\mathbb{S}^n$, and so forth. 

We will implement AI tools like TensorFlow to take advantage of the model training capabilities predicted by the proofs of \cref{top_geo_control} and \cref{top_alg_control}.

Besides setting the bases for \cref{alg:main_alg}, the family $\{f_0,f_2,\ldots,f_T\}\subset \mathbb{C}[z]$ whose existence and computability is guaranteed by \cref{top_alg_control}, provides a natural way to compress the {\bf \em mean dynamical behaviour} of a given discrete-time system $(\Sigma,\{\Theta_t\})$. 

The information compression property of $\{f_0,f_2,\ldots,f_T\}$ provides a natural connection to video streaming, this connection will be the subject of further study and experimentation. We will also explore further connections to classical and quantum finite automata.

\section{Conclusions}
\label{sec:conclusions}

Given $\varepsilon,\delta>0$, and a state $X_t$ of of discrete time almost periodic system $(\Sigma,\{\Theta_t\})$ that is $\varepsilon$-approximated by a SCL-ROM $\tilde{\Sigma}$ determined by a topological control $(\mathbb{M},Z,K,T,\varphi,\{f_t\})$ of $(\Sigma,\{\Theta_t\})$, the learning cost of a model update does not exceed the solving cost of the problem: 
\[
\mathrm{arg~min}_{p\in \mathbb{C}[z]}\{\|X_t-Kp(Z)TX_0\|~~|~~ \deg(p)\leq k\}
\]
for some $X_0\in \Sigma$, where $k$ is the control order. The application of this training technique to the extraction of (almost) periodic patterns from sampled-data discrete-time industrial systems and plants, will be further explored.

The family $\{f_1,f_2,\ldots,f_T\}\subset \mathbb{C}[z]$ derived from the implementation of \cref{alg:main_alg} provides an effective way to compress the {\bf \em mean dynamical behaviour} of a given discrete-time sampled-data system $\Sigma$.



\section*{Acknowledgments}
The numerical experiments that provided insight and motivation for the results presented in this report, together with its applications to structure preserving matrix approximation, were performed in the Scientific Computing Innovation Center ({\bf CICC-UNAH}), of the National Autonomous University of Honduras.

I am grateful with Terry Loring, Stanly Steinberg, Marc Rieffel, Concepci\'on Ferrufino, Leonel Obando, Mario Molina and William F\'unez for several interesting questions and comments, that have been very helpful for the preparation of this document.

\end{document}